\documentclass[a4paper, 12pt, leqno]{amsart}
\usepackage{amsmath}
\usepackage{amscd}
\usepackage{amssymb}
\usepackage{amsthm}
\usepackage{amsfonts}
\usepackage{latexsym}

\newtheorem{lemma}{{\sc Lemma}}[section]

\newtheorem{proposition}[lemma]{{\sc Proposition}}
\newtheorem{theorem}[lemma]{{\sc Theorem}}
\newtheorem{remark}[lemma]{{\sc Remark}}

\newtheorem{conjecture}[lemma]{{\sc Conjecture}}

\numberwithin{equation}{section}

\def\Gb{{\mathfrak{b}}}
\def\Gg{{\mathfrak{g}}}
\def\Gh{{\mathfrak{h}}}

\def\Gn{{\mathfrak{n}}}

\def\Gs{{\mathfrak{s}}}

\def\BA{{\mathbb{A}}}

\def\BC{{\mathbb{C}}}
\def\BF{{\mathbb{F}}}

\def\BQ{{\mathbb{Q}}}
\def\BZ{{\mathbb{Z}}}

\def\CB{{\mathcal B}}
\def\CC{{\mathcal C}}
\def\DD{{\mathcal D}}
\def\CO{{\mathcal O}}
\def\CI{{\mathcal I}}

\def\CK{{\mathcal K}}
\def\CL{{\mathcal L}}
\def\CM{{\mathcal M}}

\def\CR{{\mathcal R}}

\def\CV{{\mathcal V}}

\def\CZ{{\mathcal Z}}

\def\ad{\mathop{\rm ad}\nolimits}

\def\Coker{\mathop{\rm Coker}\nolimits}
\def\Comod{\mathop{\rm Comod}\nolimits}
\def\deru{\partial}

\def\End{\mathop{\rm{End}}\nolimits}

\def\For{{\mathop{\rm For}\nolimits}}

\def\Hom{\mathop{\rm Hom}\nolimits}
\def\Ht{\mathop{\rm ht}\nolimits}
\def\id{\mathop{\rm id}\nolimits}
\def\Id{\mathop{\rm Id}\nolimits}
\def\Ind{\mathop{\rm Ind}\nolimits}
\def\Image{\mathop{\rm Im}\nolimits}
\def\Ker{\mathop{\rm Ker\hskip.5pt}\nolimits}

\def\Mod{\mathop{\rm Mod}\nolimits}

\def\res{\mathop{\rm res}\nolimits}

\def\Tor{{\rm{Tor}}}

\begin{document}
\title[quantized flag manifolds]
{Differential operators on quantized flag manifolds at roots of unity II}
\dedicatory{To Etsuro Date on his 60th birthday}
\author{Toshiyuki TANISAKI}
\thanks
{
The author was partially supported by the Grants-in-Aid for Scientific Research, Challenging Exploratory Research No.\ 21654005, and
Grants-in-Aid for Scientific Research (C) No.\ 24540026 
from Japan Society for the Promotion of Science.
}
\address{
Department of Mathematics, Osaka City University, 3-3-138, Sugimoto, Sumiyoshi-ku, Osaka, 558-8585 Japan}
\email{tanisaki@sci.osaka-cu.ac.jp}
\begin{abstract}
We formulate a Beilinson-Bernstein type derived equivalence for a quantized enveloping algebra at a root of 1 as a conjecture.
It says that there exists a derived equivalence between the category of modules over a quantized enveloping algebra at a root of 1 with fixed regular Harish-Chandra central character and the category of certain twisted $D$-modules on the corresponding quantized flag manifold.
We show that the proof is reduced to a statement about the (derived) global sections of the ring of differential operators on the quantized flag manifold.
We also give a reformulation of the conjecture in terms of the (derived) induction functor.

\end{abstract}
\maketitle
\setcounter{section}{-1}
\section{Introduction}
\label{sec:Intoro}
\subsection{}
Let $G$ be a connected, simply-connected simple algebraic group over $\BC$, and let $H$ be a maximal torus of $G$.
We denote by $\Gg$ and $\Gh$ the Lie algebras of $G$ and $H$ respectively.
Let $Q$ and $\Lambda$ be the root lattice and the weight lattice respectively.
Let $h_G$ be the Coxeter number of $G$.
We fix an odd integer $\ell>h_G$, which is prime to the order of $\Lambda/Q$ and prime to 3 if $\Gg$ is of type $G_2$, $F_4$, $E_6$, $E_7$, $E_8$, and consider 
the De Concini-Kac type quantized enveloping algebra $U_\zeta$ at $q=\zeta=\exp(2\pi\sqrt{-1}/\ell)$.

In \cite{TI} we started the investigation of the corresponding quantized flag manifold $\CB_\zeta$, which is a non-commutative scheme, and the category of $D$-modules on it.
In view of a general philosophy saying that quantized objects at roots of 1 resemble ordinary objects in positive characteristics, it is natural to pursue analogue of the theory of $D$-modules on the ordinary flag manifolds in positive characteristics due to Bezrukavnikov-Mirkovi\'{c}-Rumynin \cite{BMR}.
Along this line we have established in \cite{TI} certain Azumaya properties of the ring of differential operators on the quantized flag manifold.
The aim of this paper is to investigate an analogue of another main point of \cite{BMR} about the Beilinson-Bernstein type derived equivalence.

\subsection{}
We denote by $\DD_{\CB_\zeta,1}$ the ``sheaf of rings of differential operators''on the quantized flag manifold $\CB_\zeta$.
More generally, for each $t\in H$ we have its twisted analogue denoted by $\DD_{\CB_\zeta,t}$.
It is obtained as the specialization ${\DD}_{\CB_\zeta}\otimes_{\BC[H]}\BC$ of the universally twisted sheaf  ${\DD}_{\CB_\zeta}$ with respect to the ring homomorphism $\BC[H]\to\BC$ corresponding to $t\in H$.

Let $\CB$ be the ordinary flag manifold for $G$.
Then we have a Frobenius morphism $Fr:\CB_\zeta\to\CB$, which is a finite morphism from a non-commutative scheme to an ordinary scheme.
Taking the direct images we obtain sheaves $Fr_*{\DD}_{\CB_\zeta}$, $
Fr_*{\DD}_{\CB_\zeta,t}$\,\,$(t\in H)$ of rings on $\CB$ (in the ordinary sense).
Denote by $\Mod_{coh}(Fr_*{\DD}_{\CB_\zeta,t})$ the category of coherent $Fr_*{\DD}_{\CB_\zeta,t}$-modules.
Let ${Z_{Har}(U_\zeta)}$ be the Harish-Chandra center of $U_\zeta$, and let $\BC_t$ be the corresponding one-dimensional ${Z_{Har}(U_\zeta)}$-module.
Denote by $\Mod_{f}(U_\zeta\otimes_{Z_{Har}(U_\zeta)}\BC_t)$ the category of finitely-generated $U_\zeta\otimes_{Z_{Har}(U_\zeta)}\BC_t$-modules.
Then we have a functor
\begin{equation}
\label{eq:RGamma0}
R\Gamma(\CB,\bullet):
D^b(\Mod_{coh}(Fr_*{\DD}_{\CB_\zeta,t}))
\to
D^b(\Mod_{f}(U_\zeta\otimes_{Z_{Har}(U_\zeta)}\BC_t))
\end{equation}
between derived categories.
It is natural in view of \cite{BMR} to conjecture that \eqref{eq:RGamma0}
gives an equivalence if $t$ is regular.
By imitating the argument of \cite{BMR} we can show that this is true if we have 
\begin{equation}
\label{eq:RGammaD}
R\Gamma(\CB,Fr_*{\DD}_{\CB_\zeta})\cong
U_\zeta\otimes_{Z_{Har}(U_\zeta)}\BC[\Lambda].
\end{equation}
However, we do not know how to prove \eqref{eq:RGammaD} at present, and hence we can only state it as a conjecture.
We have also a stronger conjecture 
\begin{equation}
\label{eq:RGammaDf}
R\Gamma(\CB,Fr_*({\DD}_{\CB_\zeta})_f)\cong
U_{\zeta,f}\otimes_{Z_{Har}(U_\zeta)}\BC[\Lambda],
\end{equation}
which is the analogue of  \eqref{eq:RGammaD} regarding the adjoint finite parts $({\DD}_{\CB_\zeta})_f$, $U_{\zeta,f}$ of ${\DD}_{\CB_\zeta}$, $U_\zeta$, respectively.
We will give a reformulation of \eqref{eq:RGammaDf}  in terms of the induction functor (see Conjecture \ref{conj3} below).
It turns out that \eqref{eq:RGammaDf} is equivalent to some assertions in Backelin-Kremnizer \cite{BK1}, \cite{BK2} stated to be true under certain conditions on $\ell$ (see Remark \ref{rem:BK} below).

It is also an interesting problem to find a formulation which works 
even in the case when the parameter $t\in H$  is singular.
In the case of Lie algebras in positive characteristics 
Bezrukavnikov-Mirkovi\'{c}-Rumynin \cite{BMR2} have succeeded in giving a more general framework, which works even for singular parameter, using partial flag manifolds (quotients of $G$ by parabolic subgroups).
In their case the parameter space is $\Gh^*$, and one can associate for each $h\in \Gh^*$ a parabolic subgroup whose Levi subgroup is the centralizer of $h$; however, in our case the centralizer of $t\in H$ is not necessarily a Levi subgroup of a parabolic subgroup, and hence the method in \cite{BMR2} cannot be directly applied to our case.

\subsection{}
The contents of this paper is as follows.
In Section 1 we recall basic facts on quantized enveloping algebras at roots of 1 and the corresponding quantized flag manifolds.
In Section 2 we investigate properties of the category of $D$-modules.
In particular, we show that \eqref{eq:RGammaD} implies \eqref{eq:RGamma0} for regular $t$'s, and \eqref{eq:RGammaDf} implies \eqref{eq:RGammaD}.
In Section 3 and Section 4 we recall some known results on the representations of quantized enveloping algebras and the induction functor respectively.
Finally, in Section 5 we give a reformulation of  \eqref{eq:RGammaDf} in terms of the induction functor.

\section{Quantized flag manifold}
\label{sec:QG}
\subsection{Quantized enveloping algebras}
\subsubsection{}
Let $G$ be a connected simply-connected simple algebraic group over the complex number field $\BC$.
We fix Borel subgroups $B^+$ and $B^-$ such that $H=B^+\cap B^-$ is a maximal torus of $G$.
Set $N^+=[B^+,B^+]$ and $N^-=[B^-,B^-]$.
We denote the Lie algebras of $G$, $B^+$, $B^-$, $H$, $N^+$, $N^-$ by $\Gg$, $\Gb^+$, $\Gb^-$, $\Gh$, $\Gn^+$, $\Gn^-$ respectively.
Let $\Delta\subset\Gh^*$ be the root system of $(\Gg,\Gh)$.
We denote by $\Lambda\subset\Gh^*$ and $Q\subset\Gh^*$ the weight lattice and the root lattice respectively.
For $\lambda\in\Lambda$ we denote by $\theta_\lambda$ the corresponding character of $H$.
The coordinate algebra $\BC[H]$ of $H$ is naturally identified with the group algebra $\BC[\Lambda]=\bigoplus_{\lambda\in\Lambda}\BC e(\lambda)$ via the correspondence $\theta_\lambda\leftrightarrow e(\lambda)$ for $\lambda\in\Lambda$.
We take a system of positive roots $\Delta^+$ such that $\Gb^+$ is the sum of weight spaces with weights in $\Delta^+\cup\{0\}$.
Let $\{\alpha_i\}_{i\in I}$ be the set of simple roots, and 
$\{\varpi_i\}_{i\in I}$ the corresponding set of fundamental weights.
We denote by 
$\Lambda^+$ the set of dominant integral weights.
We set $Q^+=\bigoplus_{i\in I}\BZ_{\geqq0}\alpha_i$.
Let $W\subset GL(\Gh^*)$ be the Weyl group.
For $i\in I$ we denote by $s_i\in W$ the corresponding simple reflection.
We take a $W$-invariant symmetric bilinear form
\[
(\,,\,):\Gh^*\times\Gh^*\to\BC
\]
such that $(\alpha,\alpha)=2$ for short roots $\alpha$.
For $\alpha\in\Delta$ we set $\alpha^\vee=2\alpha/(\alpha,\alpha)$.
For $i\in I$ we fix $\overline{e}_i\in\Gg_{\alpha_i}$, $\overline{f}_i\in\Gg_{-\alpha_i}$ such that $[\overline{e}_i,\overline{f}_i]=\alpha_i^\vee$ under the identification $\Gh=\Gh^*$ induced by $(\,\,,\,\,)$.

\subsubsection{}
For $n\in\BZ_{\geqq0}$ we set
\[
[n]_t=\frac{t^n-t^{-n}}{t-t^{-1}}\in\BZ[t,t^{-1}],
\qquad
[n]_t!=[n]_t[n-1]_t\cdots[2]_t[1]_t\in\BZ[t,t^{-1}].
\]

We denote by
$U_\BF$
the quantized enveloping algebra over $\BF=\BQ(q^{1/{|\Lambda/Q|}})$ associated to $\Gg$.
Namely, $U_\BF$ is the associative algebra over $\BF$ generated by elements 
\[
k_\lambda\quad(\lambda\in\Lambda),\qquad
e_i, f_i\quad( i\in I)
\]
satisfying the relations
\begin{align*}
&k_0=1,\quad 
k_\lambda k_\mu=k_{\lambda+\mu}
\qquad&(\lambda,\mu\in\Lambda),
\\
&k_\lambda e_ik_\lambda^{-1}=q^{(\lambda,\alpha_i)}e_i,\qquad
&(\lambda\in\Lambda, i\in I),
\\
&k_\lambda f_ik_\lambda^{-1}=q^{-(\lambda,\alpha_i)}f_i
\qquad
&(\lambda\in\Lambda, i\in I),
\\
&e_if_j-f_je_i=\delta_{ij}\frac{k_i-k_i^{-1}}{q_i-q_i^{-1}}
\qquad
&(i, j\in I),
\\
&\sum_{n=0}^{1-a_{ij}}(-1)^ne_i^{(1-a_{ij}-n)}e_je_i^{(n)}=0
\qquad
&(i,j\in I,\,i\ne j),
\\
&\sum_{n=0}^{1-a_{ij}}(-1)^nf_i^{(1-a_{ij}-n)}f_jf_i^{(n)}=0
\qquad
&(i,j\in I,\,i\ne j),
\end{align*}
where $q_i=q^{(\alpha_i,\alpha_i)/2}, k_i=k_{\alpha_i}, a_{ij}=2(\alpha_i,\alpha_j)/(\alpha_i,\alpha_i)$ for $i, j\in I$, and
\[
e_i^{(n)}=
e_i^n/[n]_{q_i}!,
\qquad
f_i^{(n)}=
f_i^n/[n]_{q_i}!
\]
for $i\in I$ and $n\in\BZ_{\geqq0}$.
We will use the  Hopf algebra structure of $U_\BF$ given by
\begin{align*}
&\Delta(k_\lambda)=k_\lambda\otimes k_\lambda
&(\lambda\in\Lambda),\\
&\Delta(e_i)=e_i\otimes 1+k_i\otimes e_i,\quad
\Delta(f_i)=f_i\otimes k_i^{-1}+1\otimes f_i
&(i\in I),
\\
&\varepsilon(k_\lambda)=1
,\quad
\varepsilon(e_i)=\varepsilon(f_i)=0
&(\lambda\in\Lambda, i\in I),\\
&S(k_\lambda)=k_\lambda^{-1},\quad
S(e_i)=-k_i^{-1}e_i, \quad S(f_i)=-f_ik_i
&(\lambda\in\Lambda, i\in I).
\end{align*}

Define subalgebras $U_\BF^{0}$, $U_\BF^{+}$, $U_\BF^{-}$, $U_\BF^{\geqq0}$, $U_\BF^{\leqq0}$ of $U_\BF$ by
\begin{align*}
&U_\BF^{0}=\langle k_\lambda\mid\lambda\in\Lambda\rangle,\qquad
U_\BF^+=\langle e_i\mid i\in I\rangle, \qquad
U_\BF^-=\langle f_i\mid i\in I\rangle, \\
&U_\BF^{\geqq0}=\langle k_\lambda, e_i\mid\lambda\in\Lambda, i\in I\rangle
,\qquad
U_\BF^{\leqq0}=\langle k_\lambda, f_i\mid\lambda\in\Lambda, i\in I\rangle.
\end{align*}
The multiplication of $U_\BF$ induces
isomorphisms
\begin{align}
\label{eq:tri:1}
&
U_\BF\cong U_\BF^{-}\otimes U_\BF^{0}\otimes U_\BF^{+}
\cong U_\BF^{+}\otimes U_\BF^{0}\otimes U_\BF^{-}
,\\
\label{eq:tri:2}
&
U_\BF^{\geqq0}\cong U_\BF^{0}\otimes U_\BF^{+}
\cong U_\BF^{+}\otimes U_\BF^{0},\\
\label{eq:tri:3}
&
U_\BF^{\leqq0}\cong U_\BF^{0}\otimes U_\BF^{-}
\cong U_\BF^{-}\otimes U_\BF^{0},
\end{align}
of $\BF$-modules.
\eqref{eq:tri:1} is called the triangular decomposition of $U_\BF$.
For $\gamma\in Q$ we set
\begin{align*}
U^\pm_{\BF,\gamma}&=
\{u\in U_\BF^\pm\mid k_\mu u k_{-\mu}=
q^{(\gamma,\mu)}u\quad(\mu\in\Lambda)\}.
\end{align*}
Then we have
\[
U^\pm_{\BF}=\bigoplus_{\gamma\in Q^+}U^\pm_{\BF,\pm\gamma}.
\]

For $i\in I$ we denote by $T_i$ the automorphism of the algebra $U_\BF$ given by 
\begin{align*}
&T_i(k_\mu)=k_{s_i\mu}\qquad(\mu\in\Lambda),\\
&T_i(e_j)=
\begin{cases}
\sum_{k=0}^{-a_{ij}}(-1)^kq_i^{-k}e_i^{(-a_{ij}-k)}e_je_i^{(k)}\qquad
&(j\in I,\,\,j\ne i),\\
-f_ik_i\qquad
&(j=i),
\end{cases}\\
&T_i(f_j)=
\begin{cases}
\sum_{k=0}^{-a_{ij}}(-1)^kq_i^{k}f_i^{(k)}f_jf_i^{(-a_{ij}-k)}\qquad
&(j\in I,\,\,j\ne i),\\
-k_i^{-1}e_i\qquad
&(j=i)
\end{cases}
\end{align*}
(see Lusztig \cite{Lbook}).
Let $w_0$ be the longest element of $W$.
We fix a reduced expression 
\[
w_0=s_{i_1}\cdots s_{i_N}
\]
of $w_0$, and set
\[
\beta_k=s_{i_1}\cdots s_{i_{k-1}}(\alpha_{i_k})\qquad
(1\leqq k\leqq N).
\]
Then we have $\Delta^+=\{\beta_k\mid1\leqq k\leqq N\}$.
For $1\leqq k\leqq N$ set 
\[
e_{\beta_k}=T_{i_1}\cdots T_{i_{k-1}}(e_{i_k}),\quad
f_{\beta_k}=T_{i_1}\cdots T_{i_{k-1}}(f_{i_k}).
\]
Then $\{e_{\beta_{N}}^{m_N}\cdots e_{\beta_{1}}^{m_1}\mid
m_1,\dots, m_N\geqq0\}$ (resp. 
\newline
$\{f_{\beta_{N}}^{m_N}\cdots f_{\beta_{1}}^{m_1}\mid
m_1,\dots, m_N\geqq0\}$)
is an $\BF$-basis of $U_\BF^+$ (resp. $U_\BF^-$), called the PBW-basis (see Lusztig \cite{L2}).
We have $e_{\alpha_i}=e_i$ and $f_{\alpha_i}=f_i$ for any $i\in I$.

Denote by
\begin{equation}
\label{eq:Drinfeld-pairing}
\tau: U_\BF^{\geqq0}\times U_\BF^{\leqq0}\to\BF
\end{equation}
the Drinfeld pairing.
It is characterized as the unique bilinear form satisfying
\begin{align*}
&\tau(x,y_1y_2)=(\tau\otimes\tau)(\Delta(x),y_1\otimes y_2)
&(x\in U_\BF^{\geqq0},\,y_1,y_2\in U_\BF^{\leqq0}),\\
&\tau(x_1x_2,y)=(\tau\otimes\tau)(x_2\otimes x_1,\Delta(y))
&(x_1, x_2\in U_\BF^{\geqq0},\,y\in U_\BF^{\leqq0}),\\
&\tau(k_\lambda,k_\mu)=q^{-(\lambda,\mu)}
&(\lambda,\mu\in\Lambda),\\
&\tau(k_\lambda, f_i)=\tau(e_i,k_\lambda)=0
&(\lambda\in\Lambda,\,i\in I),\\
&\tau(e_i,f_j)=\delta_{ij}/(q_i^{-1}-q_i)
&(i,j\in I)
\end{align*}
(see \cite{Lbook}, \cite{T1}).
It satisfies the following (see \cite{Lbook}, \cite{T1}).
\begin{lemma}
\label{lem:Drinfeld pairing}
\begin{itemize}
\item[\rm(i)]
$\tau(S(x),S(y))=\tau(x,y)$ for $x\in U_\BF^{\geqq0}, y\in U_\BF^{\leqq0}$.
\item[\rm(ii)]
For $x\in U_\BF^{\geqq0}, y\in U_\BF^{\leqq0}$ we have
\begin{align*}
yx=\sum_{(x)_2,(y)_2}
\tau(x_{(0)},S(y_{(0)}))\tau(x_{(2)},y_{(2)})x_{(1)}y_{(1)},\\
xy=\sum_{(x)_2,(y)_2}
\tau(x_{(0)},y_{(0)})\tau(x_{(2)},S(y_{(2)}))y_{(1)}x_{(1)}.
\end{align*}
\item[\rm(iii)]
$\tau(xk_\lambda, yk_\mu)=q^{-(\lambda,\mu)}\tau(x,y)$ for $\lambda, \mu\in\Lambda, x\in U_\BF^+, y\in U_\BF^-$.
\item[\rm(iv)]
$\tau(U^+_{\BF,\beta}, U^-_{\BF,-\gamma})=\{0\}$ for $\beta, \gamma\in Q^+$ with $\beta\ne\gamma$.
\item[\rm(v)]
For any $\beta\in Q^+$ the restriction $\tau|_{U^+_{\BF,\beta}\times U^-_{\BF,-\beta}}$ is non-degenerate.
\end{itemize}
\end{lemma}

We define an algebra homomorphism
\[
\ad:U_\BF\to\End_\BF(U_\BF)
\]
by
\[
\ad(u)(v)=\sum_{(u)}u_{(0)}v(Su_{(1)})\qquad
(u,v\in U_\BF).
\]
\subsubsection{}
We fix an integer $\ell>1$ satisfying
\begin{itemize}
\item[(a)]
$\ell$ is odd,
\item[(b)]
$\ell$ is prime to 3 if $G$ is of type $G_2$, $F_4$, $E_6$, $E_7$, $E_8$,
\item[(c)]
$\ell$ is prime to $|\Lambda/Q|$,
\end{itemize}
and a primitive $\ell$-th root $\zeta'\in\BC$ of 1.
Define a subring $\BA$ of $\BF$ by
\[
\BA=\{f(q^{1/|\Lambda/Q|})\mid
f(x)\in\BQ(x),\,\mbox{$f$ is regular at $x=\zeta'$}\}.
\]
We set $\zeta=(\zeta')^{|\Lambda/Q|}$.
We note that $\zeta$ is also a primitive $\ell$-th root of 1 by the condition (c).

We denote by $U_\BA^L$, $U_\BA$ the $\BA$-forms of $U_\BF$ called the Lusztig form and  the De Concini-Kac form respectively.
Namely, we have
\begin{align*}
U_\BA^L
&=\langle
e_i^{(m)},\, f_i^{(m)},\,k_\lambda\mid
i\in I,\,m\in\BZ_{\geqq0},\,\lambda\in\Lambda
\rangle_{\BA-{\rm alg}}
\subset U_\BF
,\\
U_\BA
&=\langle
e_i,\, f_i,\,k_\lambda\mid
i\in I,\,\lambda\in\Lambda
\rangle_{\BA-{\rm alg}}
\subset U_\BF.
\end{align*}
We have obviously
$U_\BA\subset U^L_\BA$.
The Hopf algebra structure of $U_\BF$ induces Hopf algebra structures over $\BA$ of $U_\BA^{L}$ and $U_\BA$.
We set
\begin{align*}
&U_\BA^{L,\flat}=U_\BA^{L}\cap U_\BF^{\flat},\quad
U_\BA^{\flat}=U_\BA\cap U_\BF^{\flat}
&(\flat=+, -, 0, \geqq0, \leqq0),\\
&U_{\BA,\pm\gamma}^{L,\pm}
=U^L_{\BA}\cap U_{\BF,\pm\gamma}^{\pm},\quad
U_{\BA,\pm\gamma}^{\pm}
=U_{\BA}\cap U_{\BF,\pm\gamma}^{\pm}
&(\gamma\in Q^+).
\end{align*}
Then we have triangular decompositions
\begin{align*}
&
U_\BA^{L}\cong U_\BA^{L,-}\otimes_\BA U_\BA^{L,0}\otimes_\BA U_\BA^{L,+},\\
&
U_\BA\cong U_\BA^{-}\otimes_\BA U_\BA^{0}\otimes_\BA U_\BA^{+}.
\end{align*}
Moreover,
we have
\[
U_\BA^{L,\pm}=\bigoplus_{\gamma\in Q^+}U^{L,\pm}_{\BA,\pm\gamma},
\qquad
U_\BA^\pm=\bigoplus_{\gamma\in Q^+}U^\pm_{\BA,\pm\gamma}.
\]

The Drinfeld pairing \eqref{eq:Drinfeld-pairing} induces
\begin{equation}
\label{eq:Drinfeld-pairingA}
{}^L\tau_\BA: U_\BA^{L,\geqq0}\times U_\BA^{\leqq0}\to\BA,
\qquad
\tau^L_\BA: U_\BA^{\geqq0}\times U_\BA^{L,\leqq0}\to\BA.
\end{equation}

\begin{lemma}
\label{lem:ad}
$\ad(U_\BA^L)(U_\BA)\subset U_\BA$.
\end{lemma}
\begin{proof}
It is sufficient to show 
\begin{align}
&\ad(k_\lambda)(U_\BA)\subset U_\BA\qquad(\lambda\in\Lambda),
\label{eq:lem:ad:1}\\
&\ad(e_i^{(n)})(U_\BA)\subset U_\BA
\qquad(i\in I,\,n\in\BZ_{\geqq0}),
\label{eq:lem:ad:2}\\
&\ad(f_i^{(n)})(U_\BA)\subset U_\BA
\qquad(i\in I,\,n\in\BZ_{\geqq0}).
\label{eq:lem:ad:3}
\end{align}
The proof of \eqref{eq:lem:ad:1} is easy and omitted.
By the formulas
\begin{align*}
&\ad(x)(uv)=\sum_{(x)}\ad(x_{(0)})(u)\ad(x_{(1)})(v)
&(x\in U_\BA^L, \, u,v\in U_\BA),\\
&\Delta(e_i^{(n)})=\sum_{r=0}^n
q_i^{r(n-r)}e_i^{(n-r)}k_i^r\otimes e_i^{(r)}
\qquad
&(i\in I, n\geqq0),\\
&\Delta(f_i^{(n)})=\sum_{r=0}^n
q_i^{-r(n-r)}f_i^{(r)}\otimes k_i^{-r}f_i^{(n-r)}
\qquad
&(i\in I, n\geqq0)
\end{align*}
we have only to show
\begin{align}
&\ad(e_i^{(n)})(u)\in U_\BA
&(i\in I,\,n\in\BZ_{\geqq0},\, u=k_\lambda, e_j, f_jk_j),
\label{eq:lem:ad:E}\\
&\ad(f_i^{(n)})(u)\in U_\BA
&(i\in I,\,n\in\BZ_{\geqq0}, \, u=k_\lambda, e_j, f_j).
\label{eq:lem:ad:F}
\end{align}
For $\lambda\in\Lambda$, $i, j\in I$ with $i\ne j$ and $n\in\BZ_{>0}$ we have
\begin{align*}
&\ad(e_i^{(n)})(k_\lambda)
=
\frac{(-1)^nq_i^{n(n-1)}}{[n]_{q_i}!}
\left(
\prod_{j=0}^{n-1}(q_i^{(\lambda,\alpha_i^\vee)}-q_i^{-2j})
\right)
e_i^nk_\lambda,\\
&\ad(e_i^{(n)})(e_i)
=
q_i^{-n(n+1)/2}(q_i-q_i^{-1})^ne_i^{n+1},
\\
&\ad(e_i^{(n)})(e_j)
=
\begin{cases}
\sum_{r=0}^n(-1)^rq_i^{r(n-1+a_{ij})}e_i^{(n-r)}e_je_i^{(r)}\qquad
&(n<1-a_{ij}),\\
0&(n\geqq 1-a_{ij}),
\end{cases}
\\
&\ad(e_i^{(n)})(f_ik_i)
=
\begin{cases}
(k_i^2-1)/(q_i-q_i^{-1})
&(n=1),\\
(-1)^{n-1}q_i^{(n-1)(n+2)/2}(q_i-q_i^{-1})^{n-2}e_i^{n-1}k_i^2
\quad&(n>1),
\end{cases}
\\
&\ad(e_i^{(n)})(f_jk_j)
=
0
,
\end{align*}
and hence \eqref{eq:lem:ad:E} holds
(note that $[r]_{q_i}!$ is invertible in $\BA$ for $r\leqq-a_{ij}$).
The proof of \eqref{eq:lem:ad:F} is similar and omitted.
\end{proof}

\subsubsection{}
Let us consider the specialization 
\[
\BA\to\BC\qquad
(q^{1/|\Lambda/Q|}\mapsto\zeta').
\]
Note that $q$ is mapped to $\zeta=(\zeta')^{|\Lambda/Q|}\in\BC$, which is also a primitive $\ell$-th root of 1.
We set
\begin{align*}
&U_\zeta^L=\BC\otimes_\BA U_\BA^L,\qquad
U_\zeta=\BC\otimes_\BA U_\BA,\\
&U_\zeta^{L,\flat}=\BC\otimes_\BA U_\BA^{L,\flat},\qquad
U_\zeta^{\flat}=\BC\otimes_\BA U_\BA^{\flat}
&(\flat=+,-, 0, \geqq0,\leqq0),\\
&
U_{\zeta,\pm\gamma}^{L,\pm}=\BC\otimes_\BA U_{\BA,\pm\gamma}^{L,\pm},\qquad
U_{\zeta,\pm\gamma}^{\pm}=\BC\otimes_\BA U_{\BA,\pm\gamma}^{\pm}
&(\gamma\in Q^+).
\end{align*}
Then $U^L_\zeta$ and $U_\zeta$ are Hopf algebras over $\BC$, and we have triangular decompositions
\begin{align*}
&
U^L_\zeta\cong U_\zeta^{L,-}\otimes U_\zeta^{L,0}\otimes U_\zeta^{L,+},\\
&
U_\zeta\cong U_\zeta^{-}\otimes U_\zeta^{0}\otimes U_\zeta^{+}.
\end{align*}
Moreover, we have
\[
U_\zeta^{L,\pm}=\bigoplus_{\gamma\in Q^+}U^{L,\pm}_{\zeta,\pm\gamma},
\qquad
U_\zeta^\pm=\bigoplus_{\gamma\in Q^+}U^\pm_{\zeta,\pm\gamma}.
\]

By De Concini-Kac \cite{DK} we have the following.
\begin{lemma}
\label{lem:PBW-DK}
\begin{itemize}
\item[\rm(i)]
$\{{e}_{\beta_{N}}^{m_N}\cdots {e}_{\beta_{1}}^{m_1}\mid
m_1,\dots, m_N\geqq0\}$ $($resp. 
\newline
$\{{f}_{\beta_{N}}^{m_N}\cdots {f}_{\beta_{1}}^{m_1}\mid
m_1,\dots, m_N\geqq0\}$$)$
is a ${\BC}$-basis of $U_{\zeta}^+$ $($resp. $U_{\zeta}^-$$)$.
\item[\rm(ii)]
$\{k_\lambda\mid\lambda\in\Lambda\}$ is a $\BC$-basis of $U^0_\zeta$ .
\end{itemize}
\end{lemma}
The Drinfeld pairings \eqref{eq:Drinfeld-pairingA} induce
\begin{equation}
\label{eq:Drinfeld-pairing-zeta}
{}^L\tau_\zeta: U_\zeta^{L,\geqq0}\times U_\zeta^{\leqq0}\to\BC,
\qquad
\tau^L_\zeta: U_\zeta^{\geqq0}\times U_\zeta^{L,\leqq0}\to\BC.
\end{equation}
Moreover, we have the following (see \cite[Lemma 1.5]{TI}).
\begin{proposition}
For any $\gamma\in Q^+$ the restrictions of ${}^L\tau_\zeta$ and $\tau^L_\zeta$ to
\[
U_{\zeta,\gamma}^{L,+}\times U_{\zeta,-\gamma}^{-}\to\BC,
\qquad
U_{\zeta,\gamma}^{-}\times U_{\zeta,-\gamma}^{L,-}\to\BC
\]
respectively are non-degenerate.
\end{proposition}

By Lemma \ref{lem:ad} we have an algebra homomorphism
\[
\ad:U_\zeta^L\to \End_\BC(U_\zeta).
\]

In general for a Lie algebra $\Gs$ we denote its enveloping algebra by $U(\Gs)$.
We denote by
\begin{equation}
\label{eq:LFrobenius}
\pi:U_\zeta^L\to U(\Gg)
\end{equation}
Lusztig's Frobenius homomorphism (\cite{L2}). 
Namely $\pi$ is the $\BC$-algebra homomorphism given by
\[
\pi(e_i^{(m)})=
\begin{cases}
\overline{e}_i^{(m/\ell)}\,\,&(\ell|m)\\
0&(\ell\not|m),
\end{cases}
\quad
\pi(f_i^{(m)})=
\begin{cases}
\overline{f}_i^{(m/\ell)}\,\,&(\ell|m)\\
0&(\ell\not|m),
\end{cases}
\quad
\pi(k_\lambda)=1
\]
for $i\in I$, $m\in\BZ_{\geqq0}$, $\lambda\in\Lambda$. 
Here, $\overline{e}_i^{(n)}=\overline{e}_i^n/n!$, $\overline{f}_i^{(n)}=\overline{f}_i^n/n!$ for $i\in I$ and $n\in\BZ_{\geqq0}$. 
Then $\pi$ is a homomorphism of Hopf algebras.

We recall the description of the center $Z(U_\zeta)$ of the algebra $U_\zeta$ due to De Concini-Kac \cite{DK} and De Concini-Procesi \cite{DP}.
Denote by $Z(U_{\BF})$ the center of $U_{\BF}$, and 
define a subalgebra $Z_{Har}(U_\zeta)$ of $Z(U_\zeta)$ by
\[
Z_{Har}(U_\zeta)=\Image(Z(U_{\BF})\cap U_\BA\to U_\zeta).
\]
We define a shifted action of $W$ on the group algebra 
$\BC[\Lambda]=\bigoplus_{\lambda\in\Lambda}\BC
e(\lambda)$
of $\Lambda$ by
\begin{equation}
\label{eq:shft}
w\circ e(\lambda)=\zeta^{(w\lambda-\lambda,\rho)}e(w\lambda)\qquad
(w\in W,\,\lambda\in\Lambda).
\end{equation}
Let 
\begin{equation}
\iota:Z_{Har}(U_\zeta)\to\BC[\Lambda]
\end{equation}
be the composite of
\[
Z_{Har}(U_\zeta)
\hookrightarrow
U_\zeta
\cong
U^-_\zeta\otimes
U^0_\zeta\otimes
U^+_\zeta\xrightarrow{\varepsilon\otimes1\otimes\varepsilon}
U^0_\zeta
\cong
\BC[\Lambda],
\]
where $U^0_\zeta
\cong
\BC[\Lambda]$ is given by $k_\lambda\leftrightarrow e(\lambda)$.
Then by \cite{DK} 
$\iota$ is an injective algebra homomorphism with image 
\[
\BC[2\Lambda]^{W\circ}
=\{
f\in\BC[2\Lambda]\mid w\circ f=f\quad(\forall w\in W)\}.
\]
In particular, 
we have  an isomorphism
\begin{equation}
\label{eq:HC-center}
Z_{Har}(U_\zeta)\simeq\BC[2\Lambda]^{W\circ}
\end{equation}
of $\BC$-algebras.
By \cite{DK} the elements
\[
{e_\beta}^\ell,\quad
{f_\beta}^\ell,\quad
k_{\ell\lambda}\qquad(\beta\in\Delta^+,\,\lambda\in\Lambda)
\]
are central  in $U_\zeta$.
Let $Z_{Fr}(U_\zeta)$ be the subalgebra of $U_\zeta$ generated by them.
It is a Hopf subalgebra of $U_\zeta$.
Define an algebraic subgroup $K$ of $B^+\times B^-$ by
\[
K=\{(gh, g'h^{-1})\mid
h\in H,\,g\in N^+,\,g'\in N^-\}.
\]
By \cite{DP} we have an isomorphism
\begin{equation}
\label{eq:Fr-center}
Z_{Fr}(U_\zeta)\cong \BC[K]
\end{equation}
of Hopf algebras.
We refer the reader to \cite{TI} for the explicit description of the isomorphism \eqref{eq:Fr-center}.
By \cite{DP} $Z(U_\zeta)$ is generated by $Z_{Fr}(U_\zeta)$ and $Z_{Har}(U_\zeta)$.
Moreover, we have an isomorphism 
\[
Z(U_\zeta)\cong
Z_{Har}(U_\zeta)\otimes_{Z_{Har}(U_\zeta)\cap Z_{Fr}(U_\zeta)}
Z_{Fr}(U_\zeta)
\qquad(z_1z_2\leftrightarrow z_1\otimes z_2)
\]
of algebras.

\subsection{Sheaves on quantized flag manifolds}
\subsubsection{}
We denote by $C_\BF$ the subspace of $U_\BF^*=\Hom_\BF(U_\BF,\BF)$ spanned by the matrix coefficients of finite-dimensional $U_\BF$-modules of type 1 in the sense of Lusztig, and denote by
\begin{equation}
\label{eq:Hopf-pairing}
\langle\,,\,\rangle:C_\BF\times U_\BF\to\BF
\end{equation}
the canonical pairing.
Then $C_\BF$ is endowed with a Hopf algebra structure dual to $U_\BF$ via \eqref{eq:Hopf-pairing}.
We have a $U_\BF$-bimodule structure of $C_\BF$ given by
\[
\langle u_1\cdot\varphi\cdot u_2,u\rangle
=\langle \varphi,u_2uu_1\rangle
\qquad(\varphi\in C_\BF, u, u_1, u_2\in U_\BF).
\]
Define a $\Lambda$-graded ring $A_\BF=\bigoplus_{\lambda\in\Lambda^+}A_\BF(\lambda)$ by
\begin{align*}
A_\BF&=\{\varphi\in C_\BF\mid\varphi\cdot f_i=0\quad(i\in I)\},\\
A_\BF(\lambda)&=
\{\varphi\in A_\BF\mid \varphi\cdot k_\mu=q^{(\mu,\lambda)}\varphi\quad(\mu\in\Lambda)\}.
\end{align*}
Note that $A_\BF$ is a left $U_\BF$-submodule of $C_\BF$.
For $\lambda\in\Lambda^+$ and $\xi\in\Lambda$ we set
\[
A_\BF(\lambda)_\xi
=\{\varphi\in A_\BF(\lambda)
\mid
k_\mu\cdot\varphi=q^{(\xi,\mu)}\varphi
\}.
\]
Then we have 
\[
A_\BF(\lambda)=\bigoplus_{\xi\in\lambda-Q^+}A_\BF(\lambda)_\xi.
\]

We define $\BA$-forms $C_\BA$, $A_\BA$, $A_\BA(\lambda)$ ($\lambda\in\Lambda^+$) of $C_\BF$, $A_\BF$,  $A_\BF(\lambda)$ respectively by
\begin{align*}
&C_\BA=\{\varphi\in C_\BF\mid
\langle \varphi,U_\BA^L\rangle\subset\BA\},\quad
A_\BA=A_\BF\cap C_\BA,\quad
A_\BA(\lambda)=A_\BF(\lambda)\cap C_\BA.
\end{align*}
Then $C_\BA$ is a Hopf algebra over $\BA$, and $A_\BA$ is its $\BA$-subalgebra.
Moreover, $C_\BA$ is a $U^L_\BA$-bimodule and $A_\BA$ is its left $U^L_\BA$-submodule.
We also set
$A_\BA(\lambda)_\xi=A_\BF(\lambda)_\xi\cap A_\BA$ for $\lambda\in\Lambda^+$, $\xi\in\Lambda$.

We set 
\[
C_\zeta=\BC\otimes_\BA C_\BA,\quad
A_\zeta=\BC\otimes_\BA A_\BA,\quad
A_\zeta(\lambda)=\BC\otimes_\BA A_\BA(\lambda)\qquad
(\lambda\in\Lambda^+).
\]
Then $C_\zeta$ is a Hopf algebra over $\BC$.
Moreover, the $U_\BF$-bimodule structure of $C_\BF$ induces a $U_\zeta^L$-bimodule structure of
$C_\zeta$.
For $\lambda\in\Lambda^+$ and $\xi\in\Lambda$ we set
$A_\zeta(\lambda)_\xi=\BC\otimes_\BA A_\BA(\lambda)_\xi$
Then we have 
\[
A_\zeta(\lambda)=\bigoplus_{\xi\in\lambda-Q^+}A_\zeta(\lambda)_\xi.
\]
We have a natural pairing
\begin{equation}
\label{eq:Hopf-pairing2}
\langle\,,\,\rangle:C_\zeta\times U^L_\zeta\to\BC.
\end{equation}
induced by \eqref{eq:Hopf-pairing}.

\subsubsection{}
For a ring (resp.\ $\Lambda$-graded ring) $\CR$ we denote by $\Mod(\CR)$ (resp.\ $\Mod_\Lambda(\CR)$ ) the category of $\CR$-modules (resp.\ $\Lambda$-graded left $\CR$-modules).
Assume that we are given a homomorphism $\jmath:A\to B$ of $\Lambda$-graded rings satisfying
\begin{equation}
\label{eq:cond-AK}
\jmath(A(\lambda))B(\mu)=B(\mu)\jmath(A(\lambda))\qquad
(\lambda, \mu\in\Lambda).
\end{equation}
For $M\in\Mod_{\Lambda}(B)$ let $\Tor(M)$ be the subset of $M$ consisting of $m\in M$ 
such that there exists $\lambda\in\Lambda^+$ satisfying $\jmath(A(\lambda+\mu))m=\{0\}$ for any $\mu\in\Lambda^+$.
Then $\Tor(M)$ is a subobject of $M$ in $\Mod_{\Lambda}(B)$ by \eqref{eq:cond-AK}.
We denote by $\Tor_{\Lambda^+}(A,B)$ the full subcategory of $\Mod_{\Lambda}(B)$ consisting of $M\in\Mod_{\Lambda}(B)$ such that $\Tor(M)=M$.
Note that $\Tor_{\Lambda^+}(A,B)$ is closed under taking subquotients and extensions in $\Mod_{\Lambda}(B)$.
Let $\Sigma(A,B)$ denote the collection of morphisms $f$ of $\Mod_{\Lambda}(B)$ such that its kernel $\Ker(f)$ and its cokernel $\Coker(f)$ belong to $\Tor_{\Lambda^+}(A,B)$.
Then we define an abelian category $\CC(A,B)=
\Mod_\Lambda(B)/\Tor_{\Lambda^+}(A,B)$ as
the localization 
\[
\CC(A,B)
=\Sigma(A,B)^{-1}\Mod_{\Lambda}(B)
\]
of $\Mod_{\Lambda}(B)$ with respect to the multiplicative system $\Sigma(A,B)$
(see, for example, Popescu \cite{P} for the notion of localization of categories).
We denote by 
\begin{equation}
\label{eq:omega1}
\omega(A,B)^*:\Mod_{\Lambda}(B)\to
\CC(A,B)
\end{equation}
the canonical exact functor.
It
admits a right adjoint 
\begin{equation}
\label{eq:omega2}
\omega(A,B)_*:\CC(A,B)
\to
\Mod_{\Lambda}(B),
\end{equation}
which is left exact.
It is known that $\omega(A,B)^*\circ\omega(A,B)_*\cong \Id$.
By taking the degree zero part of \eqref{eq:omega2} we obtain a left exact functor 
\begin{equation}
\label{eq:GammaAK}
\Gamma_{(A,B)}:\CC(A,B)
\to
\Mod(B(0)).
\end{equation}
The abelian category $\CC(A,B)$ has enough injectives, and we have the right derived functors
\begin{equation}
\label{eq:RGamma}
R^i\Gamma_{(A,B)}:\CC(A,B)
\to
\Mod(B(0))
\qquad(i\in\BZ)
\end{equation}
of \eqref{eq:GammaAK}.

We apply the above arguments to the case $A=B=A_\zeta$.
Then 
$\Tor(M)$ for $M\in\Mod_{\Lambda}(A_\zeta)$ consists of 
$m\in M$ such that there exists $\lambda\in\Lambda^+$ satisfying $A_\zeta(\lambda)m=\{0\}$ (see \cite[Lemma 3.4]{TI}).
We set 
\begin{equation}
\Mod(\CO_{\CB_\zeta})=
\CC(A_\zeta,A_\zeta).
\end{equation}
In this case the natural functors \eqref{eq:omega1}, \eqref{eq:omega2}, \eqref{eq:GammaAK} are simply denoted as
\begin{align}
\label{eq:omega1A}
\omega^*&:\Mod_{\Lambda}(A_\zeta)\to
\Mod(\CO_{\CB_\zeta}),\\
\omega_*&:\Mod(\CO_{\CB_\zeta})
\to
\Mod_{\Lambda}(A_\zeta),\\
\Gamma&:\Mod(\CO_{\CB_\zeta})
\to
\Mod(\BC).
\end{align}
\begin{remark}
{\rm
In the terminology of non-commutative algebraic geometry 
$\Mod(\CO_{\CB_\zeta})$ is the category of ``quasi-coherent sheaves'' on the quantized flag manifold $\CB_\zeta$, which is a ``non-commutative projective scheme''.
The notations ${\CB_\zeta}$, $\CO_{\CB_\zeta}$ have only symbolical meaning.
}
\end{remark}

\subsubsection{}
Using Lusztig's Frobenius homomorphism \eqref{eq:LFrobenius}  we will relate the quantized flag manifold $\CB_\zeta$ with the ordinary flag manifold $\CB=B^-\backslash G$.
Taking the dual Hopf algebras in \eqref{eq:LFrobenius} we obtain an injective homomorphism
$\BC[G]\to C_\zeta$
of Hopf algebras.
Moreover, its image is contained in the center of $C_\zeta$ (see Lusztig \cite{L2}).
We will regard $\BC[G]$ as a central Hopf subalgebra of $C_\zeta$ in the following.
Setting
\begin{align*}
A_1=&\{\varphi\in\BC[G]\mid
\varphi(ng)=\varphi(g)\,\,(n\in N^-,\,g\in G)\},\\
A_1(\lambda)
=&\{\varphi\in A_1\mid
\varphi(tg)=\theta_\lambda(t)\varphi(g)\,\,(t\in H,\,g\in G)\}
\qquad(\lambda\in\Lambda^+)
\end{align*}
we have a $\Lambda$-graded algebra
\[
A_1=\bigoplus_{\lambda\in\Lambda^+}A_1(\lambda).
\]
We have a left $G$-module structure of $A_1$ given by 
\[
(x\varphi)(g)=\varphi(gx)
\qquad(\varphi\in A_1, x, g\in G).
\]
In particular, $A_1$ is a $U(\Gg)$-module.
Moreover, for each $\lambda\in\Lambda^+$, $A_1(\lambda)$ is a $U(\Gg)$-submodule of $A_1$ which is an irreducible highest weight module with highest weight $\lambda$.
Regarding $\BC[G]$ as a subalgebra of $C_\zeta$ we have 
\[
A_1=A_\zeta\cap\BC[G],\qquad
A_1(\lambda)=A_\zeta(\ell\lambda)\cap\BC[G].
\]
Since the $\Lambda$-graded algebra $A_1$ is the
homogeneous coordinate algebra of the projective variety $\CB=B^-\backslash G$, we have an identification 
\begin{equation}
\label{eq:A1A1}
\Mod(\CO_\CB)=\CC(A_1,A_1)
\end{equation}
of abelian categories, where $\Mod(\CO_\CB)$ denotes the category of quasi-coherent $\CO_\CB$-modules on the ordinary flag manifold $\CB$.
We set 
\begin{equation}
\omega_{\CB*}=\omega(A_1,A_1)_*:\Mod(\CO_\CB)\to\Mod_\Lambda(A_1).
\end{equation}
For $\lambda\in\Lambda$ we denote by $\CO_\CB(\lambda)\in \Mod(\CO_\CB)$ the invertible $G$-equivariant $\CO_\CB$-module corresponding to $\lambda$.
Then under the identification \eqref{eq:A1A1} we have
\[
\omega_{\CB*}M=
\bigoplus_{\lambda\in\Lambda}\Gamma(\CB,M\otimes_{\CO_\CB}\CO_\CB(\lambda))
\qquad(M\in\Mod(\CO_\CB)),
\]
where $\Gamma(\CB,\,\,):\Mod(\CO_{\CB})\to\BC$ is the global section functor for the algebraic variety $\CB$.
In particular, the functor $\Gamma_{(A_1,A_1)}:\Mod(\CO_\CB)\to\Mod(\BC)$ is identified with 
$\Gamma(\CB,\,\,)$.

For a $\Lambda$-graded $\BC$-algebra $B$ we define a new $\Lambda$-graded  $\BC$-algebra $B^{(\ell)}$ by
\[
B^{(\ell)}(\lambda)=B(\ell\lambda)
\qquad(\lambda\in\Lambda).
\]
Let
\begin{equation}
\label{eq:(ell)}
(\,\,)^{(\ell)}:\Mod_\Lambda(B)\to\Mod_\Lambda(B^{(\ell)})
\end{equation}
be the exact functor given by 
\[
M^{(\ell)}(\lambda)=M(\ell\lambda)\qquad
(\lambda\in\Lambda)
\]
for $M\in\Mod_\Lambda(B)$.

We have the following results (\cite{TI}).
\begin{lemma}
\label{lem:Fr1}
Let $B$ be a $\Lambda$-graded $\BC$-algebra.
Assume that we are given a homomorphism $\jmath:A_\zeta\to B$ of $\Lambda$-graded $\BC$-algebras.
We denote by $\jmath':A_1\to B^{(\ell)}$ the induced homomorphism of $\Lambda$-graded $\BC$-algebras.
Assume 
\begin{align*}
\jmath(A_\zeta(\lambda)) B(\mu)&=B(\mu)\jmath(A_\zeta(\lambda))
\qquad(\lambda,\mu\in\Lambda),\\
\jmath'(A_1(\lambda))B^{(\ell)}(\mu)
&=B^{(\ell)}(\mu)\jmath'(A_1(\lambda))
\qquad(\lambda,\mu\in\Lambda).
\end{align*}
Then the exact functor 
\begin{equation*}
(\,\,)^{(\ell)}:\Mod_\Lambda(B)\to\Mod_\Lambda(B^{(\ell)})
\end{equation*} induces an equivalence
\begin{equation}
Fr_*:\CC(A_\zeta,B)
\to
\CC(A_1,B^{(\ell)})
\end{equation}
of abelian categories.
Moreover, we have
\begin{equation}
\label{eq:Frobenius}
\omega(A_1,B^{(\ell)})_{*}\circ Fr_*=(\,\,)^{(\ell)}\circ\omega(A_\zeta,B)_{*}.
\end{equation}
\end{lemma}
\begin{lemma}
\label{lem:Fr2}
Let $F$ be a $\Lambda$-graded $\BC$-algebra, and let $A_1\to F$ be a homomorphism of $\Lambda$-graded $\BC$-algebras.
Assume that $\Image(A_1\to F)$ is central in $F$.
Regard $F$ as an object of $\Mod_\Lambda(A_1)$ and consider 
$\omega_\CB^*F\in\Mod(\CO_{\CB})$.
Then the multiplication of $F$ induces an $\CO_\CB$-algebra structure of $\omega_\CB^*F$, and we have an identification
\begin{equation}
\label{eq:lem:Fr2}
\CC(A_1,F)
=
\Mod(\omega_\CB^*F),
\end{equation}
of abelian categories,
where $\Mod(\omega_\CB^*F)$ denotes the category of quasi-coherent $\omega_\CB^*F$-modules.
Moreover, under the identification \eqref{eq:lem:Fr2}
we have 
\[
\Gamma_{(A_1,F)}(M)=\Gamma(\CB,M)
\in\Mod(F(0))
\qquad(M\in\Mod(\omega_\CB^*F)).
\]
\end{lemma}

We define an $\CO_\CB$-algebra $Fr_*\CO_{\CB_\zeta}$ by
\[
Fr_*\CO_{\CB_\zeta}=\omega_\CB^*(A_\zeta^{(\ell)}).
\]
We denote by $\Mod(Fr_*\CO_{\CB_\zeta})$ the category of quasi-coherent $Fr_*\CO_{\CB_\zeta}$-modules.
By Lemma \ref{lem:Fr1} and Lemma \ref{lem:Fr2} we have the following.
\begin{lemma}
\label{lem:eq-Fr}
We have an equivalence 
\[
Fr_*:\Mod(\CO_{\CB_\zeta})\to\Mod(Fr_*\CO_{\CB_\zeta})
\]
of abelian categories.
Moreover, 
for $M\in\Mod(\CO_{\CB_\zeta})$ we have
\[
R^i\Gamma(M)\simeq R^i\Gamma(\CB,Fr_*(M)),
\]
where $\Gamma(\CB,\,\,):\Mod(\CO_{\CB})\to\Mod(\BC)$ in the right side is the global section functor for $\CB$.
\end{lemma}

\section{The category of $D$-modules}
\label{sec:D}
\subsection{Ring of differential operators}
\subsubsection{}
We define a subalgebra $D_\BF$ of $\End_\BF(A_\BF)$ by
\[
D_\BF=
\langle
\ell_\varphi, r_\varphi, \deru_u, \sigma_\lambda\mid
\varphi\in A_\BF, u\in U_\BF, \lambda\in\Lambda\rangle,
\]
where 
\[
\ell_\varphi(\psi)=\varphi\psi,\quad
r_\varphi(\psi)=\psi\varphi,\quad
\deru_u(\psi)=u\cdot\psi,\quad
\sigma_\lambda(\psi)=q^{(\lambda,\mu)}\psi
\]
for $\psi\in A_\BF(\mu)$.
In fact we have
\[
D_\BF=
\langle
\ell_\varphi, \deru_u, \sigma_\lambda\mid
\varphi\in A_\BF, u\in U_\BF, \lambda\in\Lambda\rangle
\]
by \cite[Lemma 4.1]{TI}.

We have a natural grading
\begin{align*}
&D_\BF=\bigoplus_{\lambda\in\Lambda^+}D_\BF(\lambda),\\
&D_\BF(\lambda)=\{\Phi\in D_\BF\mid
\Phi(A_\BF(\mu))\subset A_\BF(\lambda+\mu)\quad(\mu\in\Lambda)\}\qquad(\lambda\in\Lambda)
\end{align*}
of $D_\BF$.
It is easily checked that
\begin{align*}
\deru_u\ell_\varphi=&\sum_{(u)}\ell_{u_{(0)}\cdot\varphi}\deru_{u_{(1)}}
&(u\in U_\BF, \varphi\in A_\BF),\\
\deru_u\sigma_\lambda=&\sigma_\lambda\deru_u
&(u\in U_\BF, \lambda\in\Lambda),\\
\sigma_\lambda\ell_\varphi=&
q^{(\lambda,\mu)}\ell_\varphi\sigma_\lambda
&(\lambda\in\Lambda, \varphi\in A_\BF(\mu)).
\end{align*}

Set
\[
E_\BF=A_\BF\otimes U_\BF\otimes\BF[\Lambda].
\]
We have a natural $\BF$-algebra structure of $E_\BF$ such that
$A_\BF\otimes 1\otimes 1$, 
$1\otimes U_\BF\otimes 1$,
$1\otimes 1\otimes\BF[\Lambda]$
are subalgebras of $E_\BF$ naturally isomorphic to 
$A_\BF$, 
$U_\BF$,
$\BF[\Lambda]$
respectively and that we have the relations
\begin{align*}
u\varphi=&\sum_{(u)}(u_{(0)}\cdot\varphi){u_{(1)}}
&(u\in U_\BF, \varphi\in A_\BF),\\
u e(\lambda)=&e(\lambda)u
&(u\in U_\BF, \lambda\in\Lambda),\\
e(\lambda)\varphi=&
q^{(\lambda,\mu)}\varphi e(\lambda)
&(\lambda\in\Lambda, \varphi\in A_\BF(\mu))
\end{align*}
in $E_\BF$.
Here, we identify $A_\BF\otimes 1\otimes 1$, 
$1\otimes U_\BF\otimes 1$,
$1\otimes 1\otimes\BF[\Lambda]$
with 
$A_\BF$, 
$U_\BF$,
$\BF[\Lambda]$ respectively.
Then we have a surjective algebra homomorphism
\begin{equation}
\label{eq:EtoD}
E_\BF\to D_\BF
\end{equation}
sending $\varphi\in A_\BF$, $u\in U_\BF$, $e(\lambda)\in\BF[\Lambda]\,(\lambda\in\Lambda)$ to
$\ell_\varphi$, $\deru_u$, $\sigma_\lambda$ respectively.
Moreover, $E_\BF$ has an obvious $\Lambda$-grading so that \eqref{eq:EtoD} preserves the $\Lambda$-grading.

\subsubsection{}
Set
\begin{align*}
D_\BA&=
\langle
\ell_\varphi, r_\varphi, \deru_u,\sigma_\lambda\mid
\varphi\in A_\BA, u\in U_\BA, \lambda\in\Lambda\rangle_{\BA-{\rm{alg}}}\subset D_\BF,\\
E_\BA&=
A_\BA\otimes U_\BA\otimes\BA[\Lambda]
\subset E_\BA.
\end{align*}
They are $\Lambda$-graded $\BA$-subalgebras of $D_\BF$ and $E_\BF$ respectively.
Again we have
\[
D_\BA=
\langle
\ell_\varphi, \deru_u,\sigma_\lambda\mid
\varphi\in A_\BA, u\in U_\BA, \lambda\in\Lambda\rangle_{\BA-{\rm{alg}}}.
\]
by \cite{TI}.
In particular, we have a surjective homomorphism
\[
E_\BA\to D_\BA
\]
of $\Lambda$-graded algebras.
Note that there is a canonical embedding
\[
D_\BA\to\End_\BA(A_\BA).
\]

\subsubsection{}
We set
\[
D_\zeta=\BC\otimes_\BA D_\BA,\qquad
E_\zeta=\BC\otimes_\BA E_\BA=A_\zeta\otimes U_\zeta\otimes\BC[\Lambda].
\]
$D_\zeta$ is a $\Lambda$-graded $\BC$-algebra generated by elements of the form
\[
\ell_\varphi, \,\,\deru_u,\,\,\sigma_\lambda\qquad
(\varphi\in A_\zeta, \,\,u\in U_\zeta, \,\,\lambda\in\Lambda).
\]
We have a surjective homomorphism
\[
E_\zeta\to D_\zeta
\]
of $\Lambda$-graded $\BC$-algebras.

\begin{lemma}
\label{lem:HCD}
Let $z\in Z_{Har}(U_\zeta)$, and
write $\iota(z)=\sum_{\lambda\in\Lambda}c_\lambda k_{2\lambda}
\quad(c_\lambda\in\BC)$.
Then we have
\[
\deru_{z}=\sum_{\lambda\in\Lambda}c_\lambda \sigma_{2\lambda}.
\]
\end{lemma}
\begin{proof}
This follows from the corresponding statement over $\BF$, which is given in \cite{T2}
\end{proof}
\begin{remark}
{\rm
The natural algebra homomorphism
$D_\zeta\to\End_\BC(A_\zeta)$
is not injective.
}
\end{remark}

\subsubsection{}

Define an $\CO_\CB$-algebra  $Fr_*\DD_{\CB_\zeta}$ by
\[
Fr_*\DD_{\CB_\zeta}=\omega_\CB^*D_\zeta^{(\ell)}.
\]
We define 
$ZD_\zeta^{(\ell)}$ to be the central subalgebra of 
$D_\zeta^{(\ell)}$ generated by the elements of the form
\[
\ell_\varphi,\,\,
\deru_u,\,\, \sigma_\lambda\qquad
(\varphi\in A_1,\,\,u\in Z_{Fr}(U_\zeta),\,\,\lambda\in\Lambda),
\]
and set
\[
\CZ_\zeta=\omega_\CB^*ZD_\zeta^{(\ell)}.
\]
It is a central subalgebra of $Fr_*\DD_{\CB_\zeta}$.
Define a subvariety $\CV$ of $\CB\times K\times H$ by
\[
\CV=
\{
(B^-g,k,t)\in\CB\times K\times H\mid
g\kappa(k)g^{-1}\in t^{2\ell}N^-\},
\]
where $\kappa:K\to G$ is given by $\kappa(k_1,k_2)=k_1k_2^{-1}$.
We denote by 
\[
p_\CV:\CV\to\CB
\]
the projection.
Now we can state the main results of  \cite{TI}.
\begin{theorem}[\cite{TI}]
\label{thm:TI1}
The $\CO_\CB$-algebra $\CZ_\zeta$ is naturally isomorphic to $p_{\CV*}\CO_\CV$.
\end{theorem}
Define an $\CO_\CV$-algebra  $\tilde{\DD}_{\CB_\zeta}$ by
\[
\tilde{\DD}_{\CB_\zeta}
=p_{\CV}^{-1}Fr_*\DD_{\CB_\zeta}
\otimes_{p_{\CV}^{-1}p_{\CV*}\CO_\CV}
\CO_\CV.
\]
\begin{theorem}[\cite{TI}]
\label{thm:TI2}
$\tilde{\DD}_{\CB_\zeta}$ is an Azumaya algebra of rank $\ell^{2|\Delta^+|}$.
\end{theorem}
Define
\[
\delta:\CV\to K\times_{H/W} H
\]
by $\delta(B^-g,k,t)=(k,t)$, where $K\to H/W$ is given by $k\mapsto[h]$ where $h$ is an element of $H$ conjugate to the semisimple part of $\kappa(k)$, and 
$H\to H/W$ is given by $t\mapsto[t^{2\ell}]$.
\begin{theorem}[\cite{TI}]
\label{thm:TI3}
For any $(k, t)\in K\times_{H/W}H$,
the restriction of
$\tilde{\DD}_{\CB_\zeta}$ to ${\delta^{-1}(k,t)}$ is a split Azumaya algebra.
\end{theorem}

\subsection{Category of $D$-modules}

We define an abelian category $\Mod(\DD_{\CB_\zeta})$ by
\[
\Mod(\DD_{\CB_\zeta})=
\CC(A_\zeta,D_\zeta).
\]
By Lemma \ref{lem:Fr1} and Lemma \ref{lem:Fr2}
we have 
an equivalence
\begin{align}
\label{eq:equiv-D2}
&Fr_*:
\Mod(\DD_{\CB_\zeta})\to
\Mod(Fr_*\DD_{\CB_\zeta})
\end{align}
of abelian categories, 
where $\Mod(Fr_*\DD_{\CB_\zeta})$ denotes the category of quasi-coherent $Fr_*\DD_{\CB_\zeta}$-modules.
Moreover, for $M\in\Mod(\DD_{\CB_\zeta})$ we have
\begin{equation}
R^i\Gamma_{(A_\zeta,D_\zeta)}(M)= R^i\Gamma(\CB,Fr_*(M))
\in\Mod(D_\zeta(0)),
\end{equation}
where $\Gamma(\CB,\,\,)$ in the right side is the global section functor for the ordinary flag variety $\CB$.

For $t\in H$ we define an abelian category $\Mod(\DD_{\CB_\zeta,t})$ by
\[
\Mod(\DD_{\CB_\zeta,t})=
\Mod_{\Lambda, t}(D_\zeta)/(\Mod_{\Lambda,t}(D_\zeta)\cap\Tor_{\Lambda^+}(A_\zeta,D_\zeta)),
\]
where $\Mod_{\Lambda,t}(D_\zeta)$ is the full subcategory of 
$\Mod_{\Lambda}(D_\zeta)$ consisting of $M\in\Mod_{\Lambda}(D_\zeta)$ so that $\sigma_\lambda|_{M(\mu)}=\theta_\lambda(t) \zeta^{(\lambda,\mu)}\id$ for any $\lambda, \mu\in\Lambda$.
Then we can regard $\Mod(\DD_{\CB_\zeta,t})$ as a full subcategory of $\Mod(\DD_{\CB_\zeta})$
(see \cite[Lemma 4.6]{T2}).
Set
\[
Fr_*\DD_{\CB_\zeta,t}=Fr_*\DD_{\CB_\zeta}\otimes_{\BC[\Lambda]}\BC_t,
\]
where $\BC_t$ denotes the one-dimensional $\BC[\Lambda]$-module given by $e(\lambda)\mapsto\theta_\lambda(t)$ for $\lambda\in\Lambda$.
The equivalence \eqref{eq:equiv-D2} induces the equivalence
\begin{equation}
\label{eq:equiv-D3}
Fr_*:
\Mod(\DD_{\CB_\zeta,t})\to
\Mod(Fr_*\DD_{\CB_\zeta,t}),
\end{equation}
where $\Mod(Fr_*\DD_{\CB_\zeta,t})$ denotes the category of quasi-coherent $Fr_*\DD_{\CB_\zeta,t}$-modules.
In particular, for $M\in\Mod(\DD_{\CB_\zeta,t})$ we have 
\[
R^i\Gamma_{(A_\zeta,D_\zeta)}(M)
=R^i\Gamma(\CB,Fr_*M)
\in\Mod(D_{\zeta,t}(0)), 
\]
where 
$D_{\zeta,t}(0)=
D_{\zeta}(0)/\sum_{\lambda\in\Lambda}
D_{\zeta}(0)(\sigma_\lambda-\theta_\lambda(t))$.

\subsection{Conjecture}
By Lemma \ref{lem:HCD} the natural algebra homomorphism
\[
U_\zeta\otimes_\BC\BC[\Lambda]\to D_\zeta(0)
\]
factors through
\[
U_\zeta\otimes_{Z_{Har}(U_\zeta)}\BC[\Lambda]\to D_\zeta(0),
\]
where $Z_{Har}(U_\zeta)$ is identified with $\BC[2\Lambda]^{W\circ}$ 
by \eqref{eq:HC-center}.
Hence we have a natural algebra homomorphism
\begin{equation}
\label{UtoGammaD}
U_\zeta\otimes_{Z_{Har}(U_\zeta)}\BC[\Lambda]\to
\Gamma(\CB,Fr_*\DD_{\CB_\zeta}).
\end{equation}
For $t\in H$ we denote by $\BC_t$ the one-dimensional $\BC[\Lambda]$-module given by $e(\lambda)v=\theta_\lambda(t)v\,\,(v\in\BC_t)$.
Then \eqref{UtoGammaD} induces an algebra homomorphism
\begin{equation}
\label{UtoGammaD2}
U_\zeta\otimes_{Z_{Har}(U_\zeta)}\BC_t\to
\Gamma(\CB,Fr_*\DD_{\CB_\zeta,t}),
\end{equation}
where $\BC_t$ is regarded as a $Z_{Har}(U_\zeta)$-module by 
$Z_{Har}(U_\zeta)\cong\BC[2\Lambda]^{W\circ}\subset\BC[\Lambda]$. 
Denote by $h_G$ the Coxeter number for $G$.
\begin{conjecture}
\label{conj}
Assume $\ell>h_G$.
The algebra homomorphism \eqref{UtoGammaD} is an isomorphism, and we have
\[
R^i\Gamma(\CB,Fr_*\DD_{\CB_\zeta})=0
\]
for $i\ne0$.
\end{conjecture}
\begin{proposition}
Let $\ell>h_G$, and assume that Conjecture \ref{conj} is valid.
Then for $t\in H$ we have
\[
\Gamma(\CB,Fr_*\DD_{\CB_\zeta,t})\cong U_\zeta\otimes_{Z_{Har}(U_\zeta)}\BC_t,
\]
and 
\[
R^i\Gamma(\CB,Fr_*\DD_{\CB_\zeta,t})=0\qquad(i\ne0).
\]
\end{proposition}
\begin{proof}
Define $f:\CV\to H$ to be the composite of the embedding $\CV\to\CB\times K\times H$ and the projection $\CB\times K\times H\to H$ onto the third factor.
Since $p_\CV$ is an affine morphism, we have
$Rp_{\CV*}\tilde{\DD}_{\CB_\zeta}=p_{\CV*}\tilde{\DD}_{\CB_\zeta}=
Fr_*\DD_{\CB_\zeta}$.
Hence we have
\[
U_\zeta\otimes^L_{Z_{Har}(U_\zeta)}\BC[\Lambda]
=
U_\zeta\otimes_{Z_{Har}(U_\zeta)}\BC[\Lambda]
\cong R\Gamma(\CB,Fr_*\DD_{\CB_\zeta})=R\Gamma(\CV,\tilde{\DD}_{\CB_\zeta}).
\]
Here we used the fact that $\BC[\Lambda]$ is a free ${Z_{Har}(U_\zeta)}$-module (see Steinberg \cite{St}). 
Denote by $\CO_t$ the $\CO_H$-module corresponding to the $\BC[\Lambda]$-module $\BC_t$. 
Similarly we have
\begin{align*}
&Fr_*\DD_{\CB_\zeta,t}=
p_{\CV*}(\tilde{\DD}_{\CB_\zeta}\otimes_{\BC[\Lambda]}\BC_t)
=Rp_{\CV*}(\tilde{\DD}_{\CB_\zeta}\otimes_{\BC[\Lambda]}\BC_t).\end{align*}
Since $f$ is flat, we have $Lf^*\CO_t=f^*\CO_t$. Hence by Theorem \ref{thm:TI2}
we have 
\[
\tilde{\DD}_{\CB_\zeta}\otimes^L_{\CO_\CV}Lf^*\CO_t
=\tilde{\DD}_{\CB_\zeta}\otimes^L_{\CO_\CV}f^*\CO_t
=\tilde{\DD}_{\CB_\zeta}\otimes_{\CO_\CV}f^*\CO_t.
\]
It follows that
\begin{align*}
&Fr_*\DD_{\CB_\zeta,t}
=Rp_{\CV*}(\tilde{\DD}_{\CB_\zeta}\otimes^L_{\CO_\CV}Lf^*\CO_t)
=Rp_{\CV*}(\tilde{\DD}_{\CB_\zeta})\otimes^L_{\CO_H}\CO_t.
\end{align*}
Hence we have
\begin{align*}
&R\Gamma(\CB,Fr_*\DD_{\CB_\zeta,t})
=R\Gamma(H,Rf_*(\tilde{\DD}_{\CB_\zeta}\otimes^L_{\CO_\CV}Lf^*\CO_t))\\
=&R\Gamma(H,Rf_*\tilde{\DD}_{\CB_\zeta}\otimes^L_{\CO_H}\CO_t)
=R\Gamma(H,Rf_*\tilde{\DD}_{\CB_\zeta})\otimes^L_{\BC[\Lambda]}\BC_t\\
=&R\Gamma(\CV,\tilde{\DD}_{\CB_\zeta})\otimes^L_{\BC[\Lambda]}\BC_t
=U_\zeta\otimes^L_{Z_{Har}(U_\zeta)}\BC[\Lambda]
\otimes^L_{\BC[\Lambda]}\BC_t\\
=&U_\zeta\otimes^L_{Z_{Har}(U_\zeta)}\BC_t.
\end{align*}
\end{proof}

\subsection{Derived Beilinson-Bernstein equivalence}
We show that Conjecture \ref{conj} implies a variant of the Beilinson-Bernstein equivalence for derived categories.

Recall that we have an identification
\[
Z_{Har}(U_\zeta)\cong\BC[2\Lambda]^{W\circ}
\subset
\BC[2\Lambda]
\subset
\BC[\Lambda].
\]
Recall also that we identify $\BC[\Lambda]$ with the coordinate algebra $\BC[H]$ of $H$.
Set $H^{(2)}=H/\Ker(H\ni t\mapsto t^2\in H)$, and let $\pi:H\to H^{(2)}$ be the canonical homomorphism.
Then we have a natural identification $\BC[H^{(2)}]=\BC[2\Lambda]$ so that $\pi^*:\BC[H^{(2)}]\to\BC[H]$ is identified with the inclusion $\BC[2\Lambda]\subset\BC[\Lambda]$.
Denote the isomorphism $H\cong H^{(2)}$ corresponding to 
$\BC[\Lambda]\ni e(\lambda)\leftrightarrow e(2\lambda)\in \BC[2\Lambda]$ by $t\leftrightarrow t^{1/2}$.
Then we have $\pi(t)=(t^2)^{1/2}$.
The shifted action \eqref{eq:shft} of $W$ on $\BC[2\Lambda]$ induces an action of $W$ on $H^{(2)}$ given by
\[
w\circ t^{1/2}=(w(tt_{2\rho})t_{2\rho}^{-1})^{1/2}
\qquad(w\in W, t\in H),
\]
where $t_{2\rho}\in H$ is given by $\theta_\mu(t_{2\rho})=\zeta^{2(\mu,\rho)}$ for any $\mu\in\Lambda$ (note that $2(\mu,\rho)\in\BZ$), and $Z_{Har}(U_\zeta)$ is regarded as the coordinate algebra of the quotient variety $(W\circ)\backslash H^{(2)}$.
For $t\in H$ we denote by $\chi_t:\BC[\Lambda]\to\BC$ the corresponding algebra homomorphism.
By the above argument we have
\[
\chi_{t_1}|_{Z_{Har}(U_\zeta)}=
\chi_{t_2}|_{Z_{Har}(U_\zeta)}
\quad
\Longleftrightarrow
\quad
(t_1^2)^{1/2}\in W\circ t_2^{1/2}.
\]
We say that $t\in H$ is regular if 
\[
\{
w\in W\mid 
w\circ(t^2)^{1/2}=(t^2)^{1/2}\}
=\{1\}.
\]

We denote by $\Mod_{coh}(Fr_*{\DD}_{\CB_\zeta,t})$ 
(resp. 
$\Mod_{f}(U_\zeta\otimes_{Z_{Har}(U_\zeta)}\BC_t)$) the category of 
coherent $Fr_*{\DD}_{\CB_\zeta,t}$-modules (resp. 
finitely generated $U_\zeta\otimes_{Z_{Har}(U_\zeta)}\BC_t$-modules). 
We also denote by 
$\Mod_{coh,t}(Fr_*{\DD}_{\CB_\zeta})$ 
(resp.\
$\Mod_{f,t}(U_\zeta)$) the category of coherent
$Fr_*{\DD}_{\CB_\zeta}$-modules
(resp. finitely-generated $U_\zeta$-modules) 
killed by some power of the maximal ideal of  $\BC[\Lambda]$ (resp. $Z_{Har}(U_\zeta)$) corresponding to $t\in H$.

\begin{theorem}
\label{BB}
Let $\ell>h_G$, and assume that Conjecture \ref{conj} is valid.
If $t\in H$ is regular, 
then the natural functors 
\begin{gather*}
R\Gamma_{\hat{t}}:
D^b(\Mod_{coh,t}(Fr_*{\DD}_{\CB_\zeta,t}))
\to
D^b(\Mod_{f,t}(U_\zeta)),\\
R\Gamma_t:
D^b(\Mod_{coh}(Fr_*{\DD}_{\CB_\zeta,t}))
\to
D^b(\Mod_{f}(U_\zeta\otimes_{Z_{Har}(U_\zeta)}\BC_t))
\end{gather*}
give equivalences of derived categories.
\end{theorem}
The proof of this result is completely similar to that of the corresponding fact for Lie algebras in positive characteristics given in \cite{BMR}.
We only give below an outline of it.
First note the following.
\begin{proposition}
[Brown-Goodearl \cite{BGood}]
\label{prop:hom-dim}
$U_\zeta$ has finite homological dimension.
\end{proposition}
The functors 
\begin{gather*}
R\Gamma_{\hat{t}}:D^b(\Mod_{coh,t}(Fr_*{\DD}_{\CB_\zeta}))
\to
D^b(\Mod_{f,t}(U_\zeta)),\\
R\Gamma_t:D^-(\Mod_{coh}(Fr_*{\DD}_{\CB_\zeta,t}))
\to
D^-(\Mod_{f}(U_\zeta\otimes_{Z_{Har}(U_\zeta)}\BC_t))
\end{gather*}
have left adjoints
\begin{gather*}
\CL_{\hat{t}}:
D^b(\Mod_{f,t}(U_\zeta))
\to
D^b(\Mod_{coh,t}(Fr_*{\DD}_{\CB_\zeta})),\\
\CL_{t}:
D^-(\Mod_{f}(U_\zeta\otimes_{Z_{Har}(U_\zeta)}\BC_t))
\to
D^-(\Mod_{coh}(Fr_*{\DD}_{\CB_\zeta,t})).
\end{gather*}
Arguing exactly as in \cite[3.3, 3.4]{BMR} using 
Theorem \ref{thm:TI2} and 
Proposition \ref{prop:hom-dim} we obtain the following.
\begin{proposition}
\label{prop:BB}
\begin{itemize}
\item[(i)]
If $t$ is regular, the adjunction morphism
$
\Id\to R\Gamma_{\hat{t}}\circ\CL_{\hat{t}}
$
is an isomorphism on $D^b(\Mod_{f,t}(U_\zeta))$.
\item[(ii)]
For any $t$, the adjunction morphism
$
\Id\to R\Gamma_{{t}}\circ\CL_{{t}}
$
is an isomorphism on $D^-(\Mod_{f}(U_\zeta\otimes_{Z_{Har}(U_\zeta)}\BC_t))$.

\end{itemize}
\end{proposition}
Arguing exactly as in \cite[3.5]{BMR} using 
Theorem \ref{thm:TI2}, Proposition \ref{prop:BB} and
Lemma \ref{lem:symplectic} below we obtain Theorem \ref{BB}. Details are omitted.
\begin{lemma}[\cite{TM}]
\label{lem:symplectic}
The variety $\CV$ is a symplectic manifold.
\end{lemma}

\subsection{Finite part}
\subsubsection{}

In \cite{TI} we also introduced a quotient algebra $D'_\zeta$ of $E_\zeta$, which is closely related to $D_\zeta$.
Let us recall its definition.
Take bases 
$\{x_p\}_p$, $\{y_p\}_p$, $\{x^L_p\}_p$, $\{y^L_p\}_p$ of
$U^+_\zeta$, $U^-_\zeta$, $U^{L,+}_\zeta$, $U^{L,-}_\zeta$ 
respectively such that 
\[
\tau^L_\zeta(x_{p_1},y^L_{p_2})=\delta_{p_1, p_2},
\qquad
{}^L\tau_\zeta(x^L_{p_1},y_{p_2})=\delta_{p_1, p_2}.
\]
We assume that 
\[
x_p\in U^+_{\zeta,\beta_p},\quad
y_p\in U^-_{\zeta,-\beta_p},\quad
x^L_p\in U^{L,+}_{\zeta,\beta_p},\quad
y^L_p\in U^{L,-}_{\zeta,-\beta_p}
\]
for $\beta_p\in Q^+$.

For $\varphi\in A_\zeta(\lambda)_\xi$ 
with $\lambda\in\Lambda^+,\,\xi\in\Lambda$ we set
\begin{align*}
&\Omega'_1(\varphi)=
\sum_p(y_p^L\cdot\varphi)x_p\in E_{\zeta,\diamondsuit},\\
&\Omega'_2(\varphi)=
\sum_p((Sx_p^L)\cdot\varphi)y_pk_{\beta_p}k_{2\xi}e(-2\lambda)\in E_{\zeta,\diamondsuit},\\
&\Omega'(\varphi)=\Omega_1'(\varphi)-\Omega_2'(\varphi)\in E_{\zeta,\diamondsuit}.
\end{align*}
We extend $\Omega'$ to whole $A_\zeta$ by linearity.
Then $D'_\zeta$ is defined by
\begin{equation*}
D'_\zeta=E_\zeta/\sum_{\varphi\in A_
\zeta}A_\zeta\Omega'(\varphi)U_\zeta\BC[\Lambda].
\end{equation*}
We have a sequence
\[
E_\zeta\to D'_\zeta\to D_\zeta
\]
of surjective homomorphisms of $\Lambda$-graded algebras.
Moreover, $D'_\zeta\to D_\zeta$ induces an isomorphism 
\begin{equation}
\label{eq:omega-cong}
\omega^*D'_\zeta\cong\omega^* D_\zeta
\end{equation}
 in $\Mod(\CO_{\CB_\zeta})$ (see \cite[Corollary 6.6]{TI}).

\subsubsection{}
We set 
\[
U_{\BF,\diamondsuit}^{0}=\bigoplus_{\lambda\in\Lambda}\BF k_{2\lambda}
\subset U_\BF^0,\qquad
U_{\BF,\diamondsuit}=S({U}^-_\BF) U_{\BF,\diamondsuit}^{0}U_\BF^+\subset
U_\BF.
\]
Then we see easily the following.
\begin{lemma}
\label{lem:diamondsuit}
$U_{\BF,\diamondsuit}$ is an $\ad(U_\BF)$-stable subalgebra of $U_\BF$.
\end{lemma}

Set
\begin{equation}
U_{\BF,f}=\{u\in U_\BF\mid 
\dim\ad(U_\BF)(u)<\infty\}.
\end{equation}
Then $U_{\BF,f}$ is a subalgebra of $U_\BF$.
Moreover, by Joseph-Letzter \cite{JL2} we have
\begin{equation}
\label{eq:Uf}
U_{\BF,f}=\sum_{\lambda\in\Lambda^+}\ad(U_\BF)(k_{-2\lambda}),
\end{equation}
and hence $U_{\BF,f}$ is a subalgebra of $U_{\BF,\diamondsuit}$.
$U_{\BF,\diamondsuit}$ and $U_{\BF, f}$ are not Hopf subalgebras of $U_\BF$; nevertheless, they satisfy the following.
\begin{lemma}
\label{lem:DeltaAd}
We have
\[
\Delta(U_{\BF,f})\subset U_\BF\otimes U_{\BF,f},
\qquad
\Delta(U_{\BF,\diamondsuit})\subset U_\BF\otimes U_{\BF,\diamondsuit}.
\]
\end{lemma}
\begin{proof}
For $u\in U_\BF$ and $\lambda\in\Lambda^+$ we have
\begin{align*}
&\Delta(\ad(u)(k_{-2\lambda}))
=\sum_{(u)}\Delta(u_{(0)}k_{-2\lambda}(Su_{(1)})\\
=&\sum_{(u)_3}u_{(0)}k_{-2\lambda}(Su_{(3)})\otimes
u_{(1)}k_{-2\lambda}(Su_{(2)})\\
=&
\sum_{(u)_2}u_{(0)}k_{-2\lambda}(Su_{(2)})\otimes
\ad(u_{(1)})(k_{-2\lambda}).
\end{align*}
Hence the first formula follows from \eqref{eq:Uf}.
Since $U_{\BF,\diamondsuit}$ is generated by $e_i$, $S{f}_i$ for $i\in I$ and $k_{2\lambda}$ for $\lambda\in\Lambda$, 
the second formula is a consequence of the fact that $\Delta(e_i)$, $\Delta(S{f}_i)$, $\Delta(k_{2\lambda})$ belong to $U_\BF\otimes U_{\BF,\diamondsuit}$.
\end{proof}
We set
\begin{align*}
E_{\BF,\diamondsuit}&=
A_\BF\otimes U_{\BF,\diamondsuit}\otimes\BF[\Lambda]
\subset E_\BF,
\\
E_{\BF,f}&=
A_\BF\otimes
U_{\BF,f}\otimes\BF[\Lambda]
\subset E_\BF.
\end{align*}
By Lemma \ref{lem:DeltaAd} they are subalgebras of $E_\BF$.

We set 
\begin{align*}
&U_{\BA,\diamondsuit}^{0}=U_{\BF,\diamondsuit}^{0}\cap U_\BA
=\bigoplus_{\lambda\in\Lambda}\BA k_{2\lambda},\quad
U_{\BA,\diamondsuit}=
U_{\BF,\diamondsuit}\cap U_\BA
=S({U}^-_\BA) U_{\BA,\diamondsuit}^{0}U_\BA^+,\\
&U_{\BA,f}=U_{\BA}\cap U_{\BF,f},
\end{align*}
and
\begin{align*}
E_{\BA,\diamondsuit}&=E_\BA\cap E_{\BF,\diamondsuit}
=A_\BA\otimes U_{\BA,\diamondsuit}\otimes\BA[\Lambda]
\subset E_{\BF,\diamondsuit},
\\
E_{\BA,f}&=E_\BA\cap E_{\BF,f}
=A_\BA\otimes U_{\BA,f}\otimes\BA[\Lambda]
\subset E_{\BF,f},
\end{align*}

We also set 
\begin{align*}
E_{\zeta,\diamondsuit}&=
\BC\otimes_\BA E_{\BA,\diamondsuit}=
A_\zeta\otimes U_{\zeta,\diamondsuit}\otimes\BC[\Lambda]
\subset E_\zeta,
\\
E_{\zeta,f}&=
\BC\otimes_\BA E_{\BA,f}=
A_\zeta\otimes U_{\zeta,f}\otimes\BC[\Lambda]
\subset E_\zeta,
\end{align*}
and
\begin{align*}
D_{\zeta,\diamondsuit}
&=\Image(E_{\zeta,\diamondsuit}\to D_\zeta),\qquad
D_{\zeta,f}
=\Image(E_{\zeta,f}\to D_\zeta),\\
D'_{\zeta,\diamondsuit}
&=\Image(E_{\zeta,\diamondsuit}\to D'_\zeta),\qquad
D'_{\zeta,f}
=\Image(E_{\zeta,f}\to D'_\zeta).
\end{align*}
By
\[
E_\zeta\cong E_{\zeta,\diamondsuit}\otimes_{U_{\zeta,\diamondsuit}}
U_{\zeta}
\]
we obtain
\begin{align}
\label{eq:D-prime}
&D'_{\zeta,\diamondsuit}=E_{\zeta,\diamondsuit}/\sum_{\varphi\in A_
\zeta}A_
\zeta\Omega'(\varphi)U_{\zeta,\diamondsuit}\BC[\Lambda],
\\
&D'_\zeta\cong D'_{\zeta,\diamondsuit}\otimes_{U_{\zeta,\diamondsuit}}
U_{\zeta}.
\end{align}
\subsubsection{}
Since $U_{\zeta}$ is a free $U_{\zeta,\diamondsuit}$-module, we have
\[
R^i\Gamma(\omega^*D'_\zeta)
\cong
R^i\Gamma(\omega^*D'_{\zeta,\diamondsuit})
\otimes_{U_{\zeta,\diamondsuit}}
U_{\zeta}
\]
for any $i\in\BZ$.
Since $U_{\zeta,\diamondsuit}$ is a localization of 
$U_{\zeta,f}$ with respect to the Ore subset $\{k_{-2\lambda}\mid\Lambda\in\Lambda^+\}$, we have
\[
R^i\Gamma(\omega^*D'_{\zeta,\diamondsuit})
\cong
R^i\Gamma(\omega^*D'_{\zeta,f})\otimes_{U_{\zeta,f}}U_{\zeta,\diamondsuit}
\]
for any $i\in \BZ$.
It follows that 
\begin{equation}
R^i\Gamma(\omega^*D'_\zeta)
\cong
R^i\Gamma(\omega^*D'_{\zeta,f})
\otimes_{U_{\zeta,f}}
U_{\zeta}
\end{equation}
for any $i\in\BZ$.
Note 
\[
R^i\Gamma(\CB,Fr_*\DD_{\CB_\zeta})
\cong
R^i\Gamma(\omega^*D'_\zeta)
\]
by Lemma \ref{lem:eq-Fr} and \eqref{eq:omega-cong}.
Hence Conjecture \ref{conj} is a consequence of the following stronger conjecture.
\begin{conjecture}
\label{conj2}
Assume $\ell>h_G$.
We have
\[
\Gamma(\omega^*D'_{\zeta,f})\cong 
U_{\zeta,f}\otimes_{Z_{Har}(U_\zeta)}\BC[\Lambda],
\]
and
\[
R^i\Gamma(\omega^*D'_{\zeta,f})=0
\]
for $i\ne0$.
\end{conjecture}
In the rest of this paper we give a reformulation of Conjecture \ref{conj2} in terms of the induction functor.

\section{Representations}
\subsection{}
For simplicity we introduce a new notation $\tilde{U}^-_\BF=S(U^-_\BF)$.
Then we have $\tilde{U}^-_\BF=\langle \tilde{f}_i\mid i\in I\rangle$, where 
$\tilde{f}_i=f_ik_i$ for $i\in I$.
Moreover, setting 
\begin{align*}
\tilde{U}^-_{\BF,\gamma}&=
\{u\in \tilde{U}_\BF^-\mid k_\mu u k_{-\mu}=
q^{(\gamma,\mu)}u\quad(\mu\in\Lambda)\}
\end{align*}
for $\gamma\in Q$ we have
\[
\tilde{U}^-_{\BF}=\bigoplus_{\gamma\in Q^+}\tilde{U}^-_{\BF,-\gamma},
\qquad
\tilde{U}_{\BF,-\gamma}^-=
{U}_{\BF,-\gamma}^-k_\gamma\quad
(\gamma\in Q^+).
\]
We also set
\begin{align*}
\tilde{U}_\BA&=U_\BA\cap\tilde{U}_\BF,\qquad
\tilde{U}_{\BA,-\gamma}=
U_\BA\cap\tilde{U}_{\BF,-\gamma}\quad(\gamma\in Q^+),\\
\tilde{U}_\zeta&=
\BC\otimes_\BA \tilde{U}_\BA,\qquad
\tilde{U}_{\zeta,-\gamma}=
\BC\otimes_\BA \tilde{U}_{\BA,-\gamma}\quad(\gamma\in Q^+).
\end{align*}
Then we have
\[
\tilde{U}^-_{\BA}=\bigoplus_{\gamma\in Q^+}\tilde{U}^-_{\BA,-\gamma},
\qquad
\tilde{U}^-_{\zeta}=\bigoplus_{\gamma\in Q^+}\tilde{U}^-_{\zeta,-\gamma}.
\]

\subsection{}
For $\lambda\in\Lambda$ we define an algebra homomorphism $\chi_\lambda:U_\BF^0\to\BF$ by $\chi_\lambda(k_\mu)=q^{(\lambda,\mu)}\,(\mu\in\Lambda)$.
For $M\in\Mod(U_\BF)$ and $\lambda\in\Lambda$ we set
\[
M_\lambda=\{m\in M\mid
hm=\chi_\lambda(h)m\quad(h\in U_\BF^0)\}.
\]

For $\lambda\in\Lambda$ we define
$M_{+,\BF}(\lambda), M_{-,\BF}(\lambda)\in\Mod(U_\BF)$
by
\begin{align*}
M_{+,\BF}(\lambda)
=&U_\BF/
\sum_{y\in U_\BF^-}U_\BF(y-\varepsilon(y))+
\sum_{h\in U_\BF^0}U_\BF(h-\chi_\lambda(h)),\\
M_{-,\BF}(\lambda)
=&U_\BF/
\sum_{x\in U_\BF^+}U_\BF(x-\varepsilon(x))+
\sum_{h\in U_\BF^0}U_\BF(h-\chi_\lambda(h)).
\end{align*}
$M_{+,\BF}(\lambda)$ is a lowest weight module with lowest weight $\lambda$, and $M_{-,\BF}(\lambda)$ is a highest weight module with highest weight $\lambda$.
We have isomorphisms 
\[
M_{+,\BF}(\lambda)
\cong
U_\BF^{+}\quad
(\overline{u}\leftrightarrow u),\qquad
M_{-,\BF}(\lambda)
\cong
U_\BF^{-}\quad
(\overline{u}\leftrightarrow u)
\]
of $\BF$-modules.
Moreover, we have weight space decompositions
\[
M_{+,\BF}(\lambda)
=\bigoplus_{\mu\in\lambda+Q^+}M_{+,\BF}(\lambda)_\mu,\qquad
M_{-,\BF}(\lambda)
=\bigoplus_{\mu\in\lambda-Q^+}M_{-,\BF}(\lambda)_\mu.
\]

For $\lambda\in\Lambda^+$ we define 
$L_{+,\BF}(-\lambda), L_{-,\BF}(\lambda)\in\Mod_f(U_\BF)$
by
\begin{align*}
L_{+,\BF}(-\lambda)
=&U_\BF/
\sum_{y\in U_\BF^-}U_\BF(y-\varepsilon(y))+
\sum_{h\in U_\BF^0}U_\BF(h-\chi_{-\lambda}(h))
+\sum_{i\in I}U_\BF e_i^{((\lambda,\alpha_i^\vee)+1)},\\
L_{-,\BF}(\lambda)
=&U_\BF/
\sum_{x\in U_\BF^+}U_\BF(x-\varepsilon(x))+
\sum_{h\in U_\BF^0}U_\BF(h-\chi_\lambda(h))
+\sum_{i\in I}U_\BF f_i^{((\lambda,\alpha_i^\vee)+1)}.
\end{align*}
$L_{+,\BF}(-\lambda)$ is a finite-dimensional irreducible lowest weight module with lowest weight $-\lambda$, and $L_{-,\BF}(\lambda)$ is a finite-dimensional irreducible highest weight module with highest weight $\lambda$.
We have weight space decompositions
\[
L_{+,\BF}(-\lambda)
=\bigoplus_{\mu\in-\lambda+Q^+}L_{+,\BF}(-\lambda)_\mu,\qquad
L_{-,\BF}(\lambda)
=\bigoplus_{\mu\in\lambda-Q^+}L_{-,\BF}(\lambda)_\mu.
\]
For $\lambda\in\Lambda^+$ we have isomorphisms
\begin{align*}
&L_{+,\BF}(-\lambda)\cong{U}_\BF^{L,+}/
\sum_{i\in I}U_\BF^{L,+} e_i^{((\lambda,\alpha_i^\vee)+1)},
\quad(\overline{u}\leftrightarrow \overline{u})
,\\
&L_{-,\BF}(\lambda)\cong\tilde{U}_\BF^{L,-}/
\sum_{i\in I}\tilde{U}_\BF^{L,-} \tilde{f}_i^{((\lambda,\alpha_i^\vee)+1)},
\quad(\overline{u}\leftrightarrow \overline{u})
\end{align*}
of vector spaces (Lusztig \cite{L1}).

Let $M$ be a $U_\BF$-module with weight space decomposition
$M=\bigoplus_{\mu\in\Lambda}M_\mu$ such that $\dim M_\mu<\infty$ for any $\mu\in\Lambda$.
We define a  $U_\BF$-module $M^\bigstar$ by
\[
M^\bigstar=\bigoplus_{\mu\in\Lambda}
M_\mu^*\subset M^*=\Hom_\BF(M,\BF),
\]
where the action of $U_\BF$ is given by
\[
\langle um^*,m\rangle=\langle m^*,(Su)m\rangle
\qquad(u\in U_\BF, m^*\in M^\bigstar, m\in M).
\]
Here $\langle\,,\,\rangle:M^\bigstar\times M\to\BF$ is the natural pairing.

We set
\begin{align*}
M^*_{\pm,\BF}(\lambda)&=(M_{\mp,\BF}(-\lambda))^\bigstar\qquad
(\lambda\in\Lambda),\\
L^*_{\pm,\BF}(\mp\lambda)&=(L_{\mp,\BF}(\pm\lambda))^\bigstar\qquad
(\lambda\in\Lambda^+).
\end{align*}
Since $L_{\mp,\BF}(\pm\lambda)$ is irreducible we have
\[
L^*_{\pm,\BF}(\mp\lambda)\cong L_{\pm,\BF}(\mp\lambda)\qquad
(\lambda\in\Lambda^+).
\]
We define isomorphisms
\begin{equation}
\label{eq:PhiPsi}
\Phi_\lambda:{U}^{+}_\BF\to M^*_{+,\BF}(\lambda),\qquad
\Psi_\lambda:\tilde{U}^{-}_\BF\to M^*_{-,\BF}(\lambda)
\end{equation}
of vector spaces by
\begin{align*}
&\langle\Phi_\lambda(x),\overline{v}\rangle
=\tau(x, v)\qquad
(x\in U_\BF^{+}, v\in \tilde{U}_\BF^{-}),\\
&\langle\Psi_\lambda(y),\overline{Su}\rangle
=\tau(u, y)\qquad
(y\in \tilde{U}_\BF^{-}, u\in {U}_\BF^{+}).
\end{align*}

\begin{lemma}
\label{lemma:PhiPsi0}
\begin{itemize}
\item[\rm(i)]
The $U_\BF$-module structure of $M^*_{+,\BF}(\lambda)$ is given by
\begin{align}
\label{eq:dualVerma1:F}
h\cdot\Phi_\lambda(x)&=
\chi_{\lambda+\gamma}(h)
\Phi_\lambda(x)\quad
(x\in U^+_{\BF,\gamma}, h\in U^{0}_\BF),\\
\label{eq:dualVerma2:F}
v\cdot\Phi_\lambda(x)&=
\sum_{(x)}\tau(x_{(0)},Sv)\Phi_\lambda(x_{(1)})
\quad
(x\in U^+_\BF, v\in U^{-}_\BF),\\
\label{eq:dualVerma3:F}
u\cdot\Phi_\lambda(x)&=
\Phi_\lambda(
k_{-\lambda}(\ad(u)(k_\lambda xk_\lambda))k_{-\lambda})
\quad
(x\in U^+_\BF,\,\,u\in U_\BF^{+}).
\end{align}
\item[\rm(ii)]
The $U_\BF$-module structure of $M^*_{-,\BF}(\lambda)$ is given by
\begin{align}
\label{eq:dualVerma4:F}
&h\cdot\Psi_\lambda(y)=
\chi_{\lambda-\gamma}(h)
\Psi_\lambda(y)\quad(y\in \tilde{U}^-_{\BF,-\gamma}, h\in U^{0}_\BF),\\
\label{eq:dualVerma5:F}
&u\cdot\Psi_\lambda(y)=
\sum_{(y)}\tau(u,y_{(0)})\Psi_\lambda(y_{(1)})
\quad(y\in \tilde{U}^-_\BF, u\in U^{+}_\BF),\\
\label{eq:dualVerma6:F}
&v\cdot\Psi_\lambda(y)=
\Psi_\lambda(
k_\lambda(\ad(v)(k_{-\lambda} yk_{-\lambda}))k_\lambda)
\quad
(y\in \tilde{U}^-_\BF,\,\,v\in U^{-}_{\BF}).
\end{align}
\end{itemize}
\end{lemma}
\begin{proof}
We will only prove (i). 
The proof of (ii) is similar and omitted.
Note that for $x\in U^+_\BF$, $a\in U_\BF$, $v\in \tilde{U}^{-}_\BF$ we have
\[
\langle a\cdot\Phi_\lambda(x),\overline{v}\rangle
=
\langle \Phi_\lambda(x),\overline{(Sa)v}\rangle.
\]
Let us show \eqref{eq:dualVerma1:F}.
For $v\in
\tilde{U}^{-}_{\BF,-\delta}$ we have
\begin{align*}
&\langle h\cdot\Phi_\lambda(x),\overline{v}\rangle
=\langle \Phi_\lambda(x),\overline{(Sh)v}\rangle
=\delta_{\gamma,\delta}
\langle \Phi_\lambda(x),\overline{(Sh)v}\rangle\\
=&\delta_{\gamma,\delta}\chi_{\lambda+\gamma}(h)
\langle \Phi_\lambda(x),\overline{v}\rangle
=\chi_{\lambda+\gamma}(h)
\langle \Phi_\lambda(x),\overline{v}\rangle.
\end{align*}
Hence\eqref{eq:dualVerma1:F} holds.
Let us next show \eqref{eq:dualVerma2:F}.
For $v\in\tilde{U}^{-}_{\BF}$ we have
\begin{align*}
&\langle y\cdot\Phi_\lambda(x),\overline{v}\rangle
=\langle \Phi_\lambda(x),\overline{(Sy)v}\rangle
=\tau(x,(Sy)v)
=\sum_{(x)}\tau(x_{(0)},Sy)\tau(x_{(1)},v)\\
=&\left\langle
\Phi_\lambda
\left(\sum_{(x)}\tau(x_{(0)},Sy)x_{(1)}
\right),\overline{v}\right\rangle
\end{align*}
Hence \eqref{eq:dualVerma2:F} also holds.
Let us finally show \eqref{eq:dualVerma3:F}.
We may assume that $u\in U^{+}_{\BF,\beta}$ for some $\beta\in Q^+$.
Then we can write
\[
\Delta u
=\sum_j
u_jk_{\beta'_j}\otimes u'_j\quad
(\beta_j, \beta'_j\in Q^+, \beta_j+\beta'_j=\beta,\,\,
u_j\in U^+_{\BF,\beta_j},\,\,
u'_j\in U^+_{\BF,\beta'_j}).
\]
For $v\in\tilde{U}^{-}_{\BF}$ we have
\begin{align*}
&\langle u\cdot\Phi_\lambda(x),\overline{v}\rangle
=\langle \Phi_\lambda(x),\overline{(Su)v}\rangle
\\
=&\sum_{(u)_2, (v)_2}
\tau(Su_{(2)},v_{(0)})\tau(Su_{(0)},Sv_{(2)})
\langle \Phi_\lambda(x),\overline{v_{(1)}(Su_{(1)})}\rangle\\
=&
\sum_{j, (v)_2}
\tau(Su'_j,v_{(0)})\tau(u_jk_{\beta'_j},v_{(2)})
\langle \Phi_\lambda(x),\overline{v_{(1)}(Sk_{\beta'_j})}\rangle\\
=&
\sum_{j, (v)_2}
q^{(\lambda,\beta_j'-\beta_j)}
\tau(Su'_j,v_{(0)})\tau(u_jk_{\beta'_j},v_{(2)})
\langle \Phi_\lambda(x),\overline{v_{(1)}k_{-\beta_j}}\rangle\\
=&
\sum_{j, (v)_2}
q^{(\lambda,\beta_j'-\beta_j)}
\tau(Su'_j,v_{(0)})\tau(u_jk_{\beta'_j},v_{(2)})
\tau(x,{v_{(1)}k_{-\beta_j}})\\
=&
\sum_{j, (v)_2}
q^{(\lambda,\beta_j'-\beta_j)}
\tau(Su'_j,v_{(0)})
\tau(x,{v_{(1)}})
\tau(u_jk_{\beta'_j},v_{(2)})
\\
=&
\sum_{j}
q^{(\lambda,\beta_j'-\beta_j)}
\tau(u_jk_{\beta'_j}x(Su'_j),v)
\\=&
\langle\Phi_\lambda(
k_{-\lambda}(\ad(u)(k_\lambda xk_\lambda))k_{-\lambda}),\overline{v}\rangle.
\end{align*}
Here, we have used Lemma \ref{lem:Drinfeld pairing}.
Note also that
$\Delta\tilde{U}_{\BF}^-\subset
\sum_{\gamma\in Q^+}
\tilde{U}_{\BF}^-k_\gamma\otimes
\tilde{U}_{\BF,-\gamma}^-$, and hence
$\Delta_2\tilde{U}_{\BF}^-\subset
\sum_{\gamma,\delta\in Q^+}
\tilde{U}_{\BF}^-k_{\gamma+\delta}
\otimes
\tilde{U}_{\BF,-\gamma}^-k_\delta
\otimes
\tilde{U}_{\BF,-\delta}^-
$.
\eqref{eq:dualVerma3:F} is proved.
\end{proof}
For $\lambda\in\Lambda$ we denote by $\BF_\lambda^{\geqq0}=\BF1_\lambda^{\geqq0}$ 
(resp.\
$\BF_\lambda^{\leqq0}=\BF1_\lambda^{\leqq0}$)
the one-dimensional $U_\BF^{\geqq0}$-module
(resp.\
$U_\BF^{\leqq0}$-module)
such that 
$h1_\lambda^{\geqq0}=\chi_\lambda(h)1_\lambda^{\geqq0}$,
$u1_\lambda^{\geqq0}=\varepsilon(u)1_\lambda^{\geqq0}$
for $h\in U_\BF^0$ and $u\in U_\BF^+$
(resp.\
$h1_\lambda^{\leqq0}=\chi_\lambda(h)1_\lambda^{\leqq0}$,
$u1_\lambda^{\leqq0}=\varepsilon(u)1_\lambda^{\leqq0}$
for $h\in U_\BF^0$ and $u\in U_\BF^-$).

Note that for any $\lambda\in\Lambda$, $k_{-2\lambda}U^+_\BF$ (resp.\ $\tilde{U}_\BF^-k_{-2\lambda}$) is $\ad(U_\BF^{\geqq0})$-stable (resp.\ $\ad(U_\BF^{\leqq0})$-stable).
We see easily from Lemma \ref{lemma:PhiPsi0} the following.
\begin{lemma}
\label{lem:PhiPsiisom}
Let $\lambda\in\Lambda$.
\begin{itemize}
\item[(i)]
The linear map
\[
k_{-2\lambda}U^+_\BF\to
M^*_{+,\BF}(-\lambda)\otimes \BF^{\geqq0}_{\lambda}
\qquad(k_{-\lambda}xk_{-\lambda}\mapsto \Phi_{-\lambda}(x)\otimes 1^{\geqq0}_{\lambda})
\]
is an isomorphism of $U_\BF^{\geqq0}$-modules, where
$k_{-2\lambda}U^+_\BF$ is regarded as a $U_\BF^{\leqq0}$-module by the adjoint action.
\item[(ii)]
The linear map
\[
\tilde{U}_\BF^-k_{-2\lambda}\to
\BF^{\leqq0}_{-\lambda}
\otimes
M^*_{-,\BF}(\lambda)
\qquad(k_{-\lambda}yk_{-\lambda}\mapsto1^{\leqq0}_{-\lambda}\otimes\Psi_\lambda(y))
\]
is an isomorphism of $U_\BF^{\leqq0}$-modules, where
$\tilde{U}_\BF^-k_{-2\lambda}$ is regarded as a $U_\BF^{\leqq0}$-module by the adjoint action.
\end{itemize}
\end{lemma}
We have an injective $U_\BF$-homomorphism
\begin{equation}
\label{eq:LASTARtoMASTAR:F}
L^*_{\pm,\BF}(\mp\lambda)\to M^*_{\pm,\BF}(\mp\lambda)
\qquad(\lambda\in\Lambda^+).
\end{equation}
induced by the natural homomorphism
$M_{\pm,\BF}(\mp\lambda)\to L_{\pm,\BF}(\mp\lambda)$.
For $\lambda\in\Lambda^+$ we define subspaces 
$U^+_\BF(\lambda)$, $\tilde{U}_\BF^-(\lambda)$ of $U^+_\BF$, $\tilde{U}_\BF^-$ respectively by
\[
U^+_\BF(\lambda)=
\Phi_{-\lambda}^{-1}(L^*_{+,\BF}(-\lambda)),\qquad
\tilde{U}_\BF^-(\lambda)=
\Psi_\lambda^{-1}(L^*_{-,\BF}(\lambda)).
\]
\begin{lemma}
\label{lem:generating}
\begin{itemize}
\item[\rm(i)]
For $\lambda, \mu\in\Lambda^+$ we have
\[
U_\BF^+(\lambda)\subset U_\BF^+(\lambda+\mu),
\qquad
\tilde{U}^-_\BF(\lambda)\subset\tilde{U}^-_\BF(\lambda+\mu).
\]
\item[\rm(ii)]
We have
\[
U^+_\BF=\sum_{\lambda\in\Lambda^+}U^+_\BF(\lambda),
\qquad
\tilde{U}^-_\BF=\sum_{\lambda\in\Lambda^+}\tilde{U}_\BF^-(\lambda).
\]
\end{itemize}
\end{lemma}
\begin{proof}
We will only prove the statements for $U_\BF^+$.
By definition we have
$U^+_\BF(\lambda)=
\{x\in U^+_\BF\mid
\tau(x,I_\lambda)=\{0\}\}$, where
$
I_\lambda=\sum_{i\in I}\tilde{U}^-_\BF
\tilde{f}_i^{((\lambda,\alpha_i^\vee)+1)}.
$
Hence (i) is a consequence of $I_\lambda\supset I_{\lambda+\mu}$ for $\lambda, \mu\in\Lambda^+$.
To show (ii) it is sufficient to show that for any $\beta\in Q^+$ there exists some $\lambda\in\Lambda^+$ such that $U^+_{\BF,\beta}\subset U_\BF^+(\lambda)$.
Set $m=\Ht(\beta)$.
If $\lambda\in\Lambda^+$ satisfies $(\lambda,\alpha_i^\vee)\geqq m$ for any $i\in I$, then we have
$I_\lambda\subset \bigoplus_{\gamma\in Q^+,\Ht(\gamma)>m}
\tilde{U}^-_{\BF,-\gamma}$.
From this we obtain
$\tau(U^+_{\BF,\beta},I_\lambda)=\{0\}$, and hence $U^+_{\BF,\beta}\subset U_\BF^+(\lambda)$.
\end{proof}
\begin{lemma}
\label{lem:lambda-finite}
For $\lambda\in\Lambda^+$ we have
\[
\tilde{U}^-_\BF(\lambda)k_{-2\lambda}\subset U_{\BF,f},\qquad
k_{-2\lambda}{U}^+_\BF(\lambda)\subset U_{\BF,f}
\]
\end{lemma}
\begin{proof}
By Lemma \ref{lem:PhiPsiisom} 
we have an isomorphism 
\[
k_{-2\lambda}U^+_\BF(\lambda)\to
L^*_{+,\BF}(-\lambda)\otimes \BF^{\geqq0}_{\lambda}
\qquad(k_{-\lambda}xk_{-\lambda}\mapsto \Phi_{-\lambda}(x)\otimes 1^{\geqq0}_{\lambda})
\]
of $U_\BF^{\geqq0}$-modules.
We have $L^*_{+,\BF}(-\lambda)\cong L_{+,\BF}(-\lambda)$ and hence 
$L^*_{+,\BF}(-\lambda)\otimes \BF^{\geqq0}_{\lambda}$ is generated by $\Phi_{-\lambda}(1)\otimes1^{\geqq0}_{\lambda}$ as a $U_\BF^{\geqq0}$-module.
It follows that
\[
k_{-2\lambda}U^+_\BF(\lambda)
=\ad(U_\BF^{\geqq0})(k_{-2\lambda})\subset U_{\BF,f}
\]
by \eqref{eq:Uf}.
The proof of $\tilde{U}^-_\BF(\lambda)k_{-2\lambda}\subset U_{\BF,f}$ is similar.
\end{proof}

\subsection{}
It is well-known that for $\lambda, \mu\in\Lambda$ such that $\lambda\ne\mu$ there exists
$h\in U^{L,0}_\BA$ such that
$\chi_\lambda(h)=1$ and $\chi_\mu(h)=0$.
In particular, we have
$\chi_\lambda\ne\chi_\mu$ (see for example \cite[Lemma 2.3]{TI}).

For $M\in\Mod(U^L_\BA)$ and $\lambda\in\Lambda$ we set
\[
M_\lambda=\{m\in M\mid
hm=\chi_\lambda(h)m\quad(h\in U^{L,0}_\BA)\}.
\]
For $\lambda\in\Lambda$ we define 
$M_{+,\BA}(\lambda), M_{-,\BA}(\lambda)\in\Mod(U^L_\BA)$
by
\begin{align*}
M_{+,\BA}(\lambda)
=&U^L_\BA/
\sum_{y\in U_\BA^{L,-}}U^L_\BA(y-\varepsilon(y))+
\sum_{h\in U_\BA^{L,0}}U^L_\BA(h-\chi_\lambda(h)),\\
M_{-,\BA}(\lambda)
=&U^L_\BA/
\sum_{x\in U_\BA^{L,+}}U^L_\BA(x-\varepsilon(x))+
\sum_{h\in U_\BA^{L,0}}U^L_\BA(h-\chi_\lambda(h)).
\end{align*}
By the triangular decomposition we have isomorphisms
\[
M_{+,\BA}(\lambda)
\cong
U_\BA^{L,+}\quad
(\overline{u}\leftrightarrow u),\qquad
M_{-,\BA}(\lambda)
\cong
U_\BA^{L,-}\quad
(\overline{u}\leftrightarrow u)
\]
of $\BA$-modules.
In particular, $M_{\pm,\BA}(\lambda)$ is a free $\BA$-module and we have $\BF\otimes_\BA M_{\pm,\BA}(\lambda)\cong M_{\pm,\BF}(\lambda)$.
Moreover, we have weight space decompositions
\[
M_{+,\BA}(\lambda)
=\bigoplus_{\mu\in\lambda+Q^+}M_{+,\BA}(\lambda)_\mu,\qquad
M_{-,\BA}(\lambda)
=\bigoplus_{\mu\in\lambda-Q^+}M_{-,\BA}(\lambda)_\mu.
\]

For $\lambda\in\Lambda^+$ we define 
$L_{+,\BA}(-\lambda)\in\Mod(U^L_\BA)$ (resp. $L_{-,\BA}(\lambda)\in\Mod(U^L_\BA)$)
to be the $U^L_\BA$-submodule of
$L_{+,\BF}(-\lambda)$ (resp. $L_{-,\BF}(\lambda)$) 
generated by $\overline{1}\in L_{+,\BF}(-\lambda)$ (resp. $\overline{1}\in L_{-,\BF}(\lambda)$).
By definition $L_{\pm,\BA}(\mp\lambda)$ is a free $\BA$-module and we have $\BF\otimes_\BA L_{\pm,\BA}(\mp\lambda)\cong L_{\pm,\BF}(\mp\lambda)$.
Moreover, we have weight space decompositions
\[
L_{+,\BA}(-\lambda)
=\bigoplus_{\mu\in-\lambda+Q^+}L_{+,\BA}(-\lambda)_\mu,\qquad
L_{-,\BA}(\lambda)
=\bigoplus_{\mu\in\lambda-Q^+}L_{-,\BA}(\lambda)_\mu.
\]
The canonical surjective $U_\BF$-homomorphism $M_{\pm,\BF}(\mp\lambda)\to L_{\pm,\BF}(\mp\lambda)$ induces a surjective $U_\BA^L$-homomorphism 
\begin{equation}
\label{eq:MAtoLA}
M_{\pm,\BA}(\mp\lambda)\to L_{\pm,\BA}(\mp\lambda)
\qquad(\lambda\in\Lambda^+).
\end{equation}
Note that \eqref{eq:MAtoLA} a split epimorphism of $\BA$-modules since $\BA$ is PID, and $M_{\pm,\BA}(\mp\lambda)_\mu$, $L_{\pm,\BA}(\mp\lambda)_\mu$ are torsion free finitely generated $\BA$-modules for each $\mu\in\Lambda$.

Let $M$ be a $U^L_\BA$-module with weight space decomposition
$M=\bigoplus_{\mu\in\Lambda}M_\mu$ such that $M_\mu$ is a free $\BA$-module of finite rank for any $\mu\in\Lambda$.
We define a  $U^L_\BA$-module $M^\bigstar$ by
\[
M^\bigstar=\bigoplus_{\mu\in\Lambda}
\Hom_\BA(M_\mu,\BA)\subset \Hom_\BA(M,\BA),
\]
where the action of $U^L_\BA$ is given by
\[
\langle um^*,m\rangle=\langle m^*,(Su)m\rangle
\qquad(u\in U^L_\BA, m^*\in M^\bigstar, m\in M).
\]
Here $\langle\,,\,\rangle:M^\bigstar\times M\to\BA$ is the natural pairing.

We set
\begin{align*}
M^*_{\pm,\BA}(\lambda)&=(M_{\mp,\BA}(-\lambda))^\bigstar\qquad
(\lambda\in\Lambda),\\
L^*_{\pm,\BA}(\mp\lambda)&=(L_{\mp,\BA}(\pm\lambda))^\bigstar\qquad
(\lambda\in\Lambda^+).
\end{align*}
Then $M^*_{\pm,\BA}(\lambda)$ for $\lambda\in\Lambda$ and $L^*_{\pm,\BA}(\mp\lambda)$ for $\lambda\in\Lambda^+$ are free $\BA$-modules satisfying 
\[
\BF\otimes_\BA M^*_{\pm,\BA}(\lambda)
\cong
M^*_{\pm,\BF}(\lambda),\qquad
\BF\otimes_\BA L^*_{\pm,\BA}(\mp\lambda)
\cong
L^*_{\pm,\BF}(\mp\lambda).
\]
Moreover, we can identify $M^*_{\pm,\BA}(\lambda)$ and $L^*_{\pm,\BA}(\mp\lambda)$ with $\BA$-submodules of 
$M^*_{\pm,\BF}(\lambda)$ and $L^*_{\pm,\BF}(\mp\lambda)$ respectively.
Under this identification we have 
\begin{equation}
\label{eq:MLAF}
L^*_{\pm,\BA}(\mp\lambda)=
L^*_{\pm,\BF}(\mp\lambda)\cap M^*_{\pm,\BA}(\mp\lambda)\qquad(\lambda\in\Lambda^+).
\end{equation}
In particular, the $U_\BA^L$-homomorphism
\begin{equation}
\label{eq:LASTARtoMASTAR}
L^*_{\pm,\BA}(\mp\lambda)\to M^*_{\pm,\BA}(\mp\lambda)
\qquad(\lambda\in\Lambda^+).
\end{equation}
is a split monomorphism of $\BA$-modules.

By abuse of notation we write
\begin{equation}
\label{eq:PhiPsi:A}
\Phi_\lambda:{U}^{+}_\BA\to M^*_{+,\BA}(\lambda),\qquad
\Psi_\lambda:\tilde{U}^{-}_\BA\to M^*_{-,\BA}(\lambda)
\end{equation}
the isomorphisms of $\BA$-modules induced by \eqref{eq:PhiPsi}.
By Lemma \ref{lemma:PhiPsi0} we have the following.
\begin{lemma}
\label{lemma:PhiPsi1}
\begin{itemize}
\item[\rm(i)]
The $U^L_\BA$-module structure of $M^*_{+,\BA}(\lambda)$ is given by
\begin{align}
\label{eq:dualVerma1:A}
h\cdot\Phi_\lambda(x)&=
\chi_{\lambda+\gamma}(h)
\Phi_\lambda(x)\quad
(x\in U^+_{\BA,\gamma}, h\in U^{L,0}_\BA),\\
\label{eq:dualVerma2:A}
v\cdot\Phi_\lambda(x)&=
\sum_{(x)}\tau_\BA^L(x_{(0)},Sv)\Phi_\lambda(x_{(1)})
\quad
(x\in U^+_\BA, v\in U^{L,-}_\BA),\\
\label{eq:dualVerma3:A}
u\cdot\Phi_\lambda(x)&=
\Phi_\lambda(
k_{-\lambda}(\ad(u)(k_\lambda xk_\lambda))k_{-\lambda})
\quad
(x\in U^+_\BA,\,\,u\in U_\BA^{L,+}).
\end{align}
\item[\rm(ii)]
The $U^L_\BA$-module structure of $M^*_{-,\BA}(\lambda)$ is given by
\begin{align}
\label{eq:dualVerma4:A}
&h\cdot\Psi_\lambda(y)=
\chi_{\lambda-\gamma}(h)
\Psi_\lambda(y)\quad(y\in \tilde{U}^-_{\BA,-\gamma}, h\in U^{L,0}_\BA),\\
\label{eq:dualVerma5:A}
&u\cdot\Psi_\lambda(y)=
\sum_{(y)}{}^L\tau_\BA(u,y_{(0)})\Psi_\lambda(y_{(1)})
\quad(y\in \tilde{U}^-_\BA, u\in U^{L,+}_\BA),\\
\label{eq:dualVerma6:A}
&v\cdot\Psi_\lambda(y)=
\Psi_\lambda(
k_\lambda(\ad(v)(k_{-\lambda} yk_{-\lambda}))k_\lambda)
\quad
(y\in \tilde{U}^-_\BA,\,\,v\in U^{L,-}_{\BA}).
\end{align}
\end{itemize}
\end{lemma}
For $\lambda\in\Lambda^+$ we define $\BA$-submodules 
$U^+_\BA(\lambda)$, $\tilde{U}_\BA^-(\lambda)$ of $U^+_\BA$, $\tilde{U}_\BA^-$ respectively by
\[
U^+_\BA(\lambda)=
\Phi_{-\lambda}^{-1}(L^*_{+,\BA}(-\lambda)),\qquad
\tilde{U}_\BA^-(\lambda)=
\Psi_\lambda^{-1}(L^*_{-,\BA}(\lambda)).
\]
The embeddings
\begin{equation}
\label{eq:Upmlambda}
U^+_\BA(\lambda)\hookrightarrow  U^+_\BA,\qquad
\tilde{U}_\BA^-(\lambda)\hookrightarrow\tilde{U}_\BA^-
\qquad(\lambda\in\Lambda^+)
\end{equation}
are split monomorphisms of $\BA$-modules.
By \eqref{eq:MLAF} we have
\begin{equation}
\label{eq:MLAF2}
U_\BA^+(\lambda)
=U_\BF^+(\lambda)
\cap U^+_\BA,\qquad
\tilde{U}_\BA^-(\lambda)
=\tilde{U}_\BF^-(\lambda)
\cap \tilde{U}^-_\BA
\qquad(\lambda\in\Lambda^+).
\end{equation}
In particular, we have
\begin{align}
\label{eq:generating:A:1}
&U_\BA^+(\lambda)\subset U_\BA^+(\lambda+\mu),
\quad\tilde{U}^-_\BA(\lambda)\subset\tilde{U}^-_\BA(\lambda+\mu)
&(\lambda, \mu\in\Lambda^+),\\
\label{eq:generating:A:2}
&U^+_\BA=\sum_{\lambda\in\Lambda^+}U^+_\BA(\lambda),
\quad
\tilde{U}^-_\BA=\sum_{\lambda\in\Lambda^+}\tilde{U}_\BA^-(\lambda),\\
&
\tilde{U}^-_\BA(\lambda)k_{-2\lambda}\subset U_{\BA,f},\quad
k_{-2\lambda}{U}^+_\BA(\lambda)\subset U_{\BA,f}&(\lambda\in\Lambda^+)
\label{eq:generating:A:3}
\end{align}
by Lemma \ref{lem:generating} and Lemma \ref{lem:lambda-finite}.

\subsection{}
Let $\lambda\in\Lambda$.
By abuse of notation 
we also denote by $\chi_\lambda:U^{L,0}_\zeta\to\BC$ the $\BC$-algebra homomorphism induced by $\chi_\lambda:U^{L,0}_\BA\to\BA$.
Then 
$\{\chi_\lambda\}_{\lambda\in\Lambda}$ is a linearly independent subset of the $\BC$-module $\Hom_\BC(U^{L,0}_\zeta,\BC)$.
For $M\in\Mod(U^L_\zeta)$ and $\lambda\in\Lambda$ we set
\[
M_\lambda=\{m\in M\mid hm=\chi_\lambda(h)m\,\,(h\in U^{L,0}_\zeta)\}.
\]

For $\lambda\in\Lambda$ we set
\[
M_{\pm,\zeta}(\lambda)=\BC\otimes_\BA M_{\pm,\BA}(\lambda),
\qquad
M^*_{\pm,\zeta}(\lambda)=\BC\otimes_\BA M^*_{\pm,\BA}(\lambda).
\]
For $\lambda\in\Lambda^+$ we set
\[
L_{\pm,\zeta}(\mp\lambda)=\BC\otimes_\BA L_{\pm,\BA}(\mp\lambda),
\qquad
L^*_{\pm,\zeta}(\mp\lambda)=\BC\otimes_\BA L^*_{\pm,\BA}(\mp\lambda).
\]
We have canonical $U^L_\zeta$-homomorphisms
\begin{align}
\label{eq:MtoL}
M_{\pm,\zeta}(\mp\lambda)\to L_{\pm,\zeta}(\mp\lambda)
\qquad(\lambda\in\Lambda^+),
\\
\label{eq:LSTARtoMSTAR}
L^*_{\pm,\zeta}(\mp\lambda)\to M^*_{\pm,\zeta}(\mp\lambda)
\qquad(\lambda\in\Lambda^+).
\end{align}
Note that \eqref{eq:MtoL} is surjective, and \eqref{eq:LSTARtoMSTAR} is injective.

For any $\lambda\in\Lambda^+$ we have an isomorphism
\begin{equation}
\label{eq:AcongL}
A_\zeta(\lambda)\cong L^*_{-,\zeta}(\lambda)
\end{equation}
of $U_\zeta^L$-modules (see, for example, \cite{TI}).

Let $\lambda\in\Lambda$.
By abuse of notation we also denote by
\[
\Phi_\lambda:{U}^{+}_\zeta\to M^*_{+,\zeta}(\lambda),\qquad
\Psi_\lambda:\tilde{U}^{-}_\zeta\to M^*_{-,\zeta}(\lambda)
\]
the isomorphisms of $\BC$-modules given by
\begin{align*}
&\langle\Phi_\lambda(x),\overline{v}\rangle
=\tau^L_\zeta(x, v)\qquad
(x\in U_\zeta^{+}, v\in \tilde{U}_\zeta^{L,-}),\\
&\langle\Psi_\lambda(y),\overline{Su}\rangle
={}^L\tau_\zeta(u, y)\qquad
(y\in \tilde{U}_\zeta^{-}, u\in {U}_\zeta^{L,+}).
\end{align*}
By Lemma \ref{lemma:PhiPsi1} we have the following.
\
\begin{lemma}
\label{lemma:PhiPsi2}
\begin{itemize}
\item[\rm(i)]
The $U^L_\zeta$-module structure of $M^*_{+,\zeta}(\lambda)$ is given by
\begin{align}
\label{eq:dualVerma1}
h\cdot\Phi_\lambda(x)&=
\chi_{\lambda+\gamma}(h)
\Phi_\lambda(x)\quad
(x\in U^+_{\zeta,\gamma}, h\in U^{L,0}_\zeta),\\
\label{eq:dualVerma2}
v\cdot\Phi_\lambda(x)&=
\sum_{(x)}\tau^L_\zeta(x_{(0)},Sv)\Phi_\lambda(x_{(1)})
\quad
(x\in U^+_\zeta, v\in U^{L,-}_\zeta),\\
\label{eq:dualVerma3}
u\cdot\Phi_\lambda(x)&=
\Phi_\lambda(
k_{-\lambda}(\ad(u)(k_\lambda xk_\lambda))k_{-\lambda})
\quad
(x\in U^+_\zeta,\,\,u\in U_\zeta^{L,+}).
\end{align}
\item[\rm(ii)]
The $U^L_\zeta$-module structure of $M^*_{-,\zeta}(\lambda)$ is given by
\begin{align}
\label{eq:dualVerma4}
&h\cdot\Psi_\lambda(y)=
\chi_{\lambda-\gamma}(h)
\Psi_\lambda(y)\quad(y\in \tilde{U}^-_{\zeta,-\gamma}, h\in U^{L,0}_\zeta),\\
\label{eq:dualVerma5}
&u\cdot\Psi_\lambda(y)=
\sum_{(y)}{}^L\tau_\zeta(u,y_{(0)})\Psi_\lambda(y_{(1)})
\quad(y\in \tilde{U}^-_\zeta, u\in U^{L,+}_\zeta),\\
\label{eq:dualVerma6}
&v\cdot\Psi_\lambda(y)=
\Psi_\lambda(
k_\lambda(\ad(v)(k_{-\lambda} yk_{-\lambda}))k_\lambda)
\quad
(y\in \tilde{U}^-_\zeta,\,\,v\in U^{L,-}_{\zeta}).
\end{align}
\end{itemize}
\end{lemma}
For $\lambda\in\Lambda^+$ 
we set
\[
U^+_\zeta(\lambda)
=\BC\otimes_\BA U^+_\BA(\lambda),\quad
\tilde{U}^-_\zeta(\lambda)
=\BC\otimes_\BA \tilde{U}^-_\BA(\lambda).
\]
Then 
$U^+_\zeta(\lambda)$ and $\tilde{U}_\zeta^-(\lambda)$ are the  $\BC$-submodules of $U^+_\zeta$ and $\tilde{U}_\zeta^-$ respectively satisfying 
$\Phi_{-\lambda}(U^+_\zeta(\lambda))=L^*_{+,\zeta}(-\lambda)$ and
$\Psi_\lambda(\tilde{U}^-_\zeta(\lambda))=L^*_{-,\zeta}(\lambda)$.
We have linear isomorphisms
\begin{equation}
\label{eq:PhiPsiLzeta}
\Phi_{-\lambda}:{U}^{+}_\zeta(\lambda)\to L^*_{+,\zeta}(-\lambda),\quad
\Psi_\lambda:\tilde{U}^{-}_\zeta(\lambda)\to L^*_{-,\zeta}(\lambda)
\quad(\lambda\in\Lambda^+).
\end{equation}
By \eqref{eq:generating:A:1}, \eqref{eq:generating:A:2}, \eqref{eq:generating:A:3} we have
\begin{align}
\label{eq:generating:zeta:1}
&U_\zeta^+(\lambda)\subset U_\zeta^+(\lambda+\mu),
\quad\tilde{U}^-_\zeta(\lambda)\subset\tilde{U}^-_\zeta(\lambda+\mu)
&(\lambda, \mu\in\Lambda^+),\\
\label{eq:generating:zeta:2}
&U^+_\zeta=\sum_{\lambda\in\Lambda^+}U^+_\zeta(\lambda),
\quad
\tilde{U}^-_\zeta=\sum_{\lambda\in\Lambda^+}\tilde{U}_\zeta^-(\lambda),\\
&
\label{eq:generating:zeta:3}
\tilde{U}^-_\zeta(\lambda)k_{-2\lambda}\subset U_{\zeta,f},\quad
k_{-2\lambda}{U}^+_\zeta(\lambda)\subset U_{\BA,f}&(\lambda\in\Lambda^+).
\end{align}
By \eqref{eq:generating:zeta:2}, \eqref{eq:generating:zeta:3} we see easily the following.
\begin{lemma}
\label{lem:localization}
For any $u\in U_\zeta$ there exists some $\lambda\in\Lambda^+$ such that 
$uk_{-2\lambda}\in U_{\zeta, f}$.
\end{lemma}

\section{Induction functor}
\label{sec:Ind}
\subsection{Functors}
We set
\[
C_\zeta^{\leqq0}=C_\zeta/I,\qquad
I=\{
\varphi\in C_\zeta\mid
\langle\varphi,U_\zeta^{L,\leqq0}\rangle=\{0\}\}.
\]
Then $C_\zeta^{\leqq0}$ is a Hopf algebra and we have a Hopf pairing
\[
\langle\,,\,\rangle:
C_\zeta^{\leqq0}\times U_\zeta^{L,\leqq0}\to\BC.
\]
We have a canonical Hopf algebra homomorphism
\[
\res:C_\zeta\to C_\zeta^{\leqq0}.
\]

Following Backelin-Kremnizer \cite{BK1}
we define abelian categories $\CM_\zeta$ and $\CM_\zeta^{eq}$ as follows.

An object of $\CM_\zeta$ is a triplet $(M,\alpha,\beta)$ with
\begin{itemize}
\item[(1)]
$M$ is a vector space over $\BC$, 
\item[(2)]
$\alpha:C_\zeta\otimes M\to M$ is a left $C_\zeta$-module structure of $M$, 
\item[(3)]
$\beta:M\to C_\zeta^{\leqq0}\otimes M$ is a left $C_\zeta^{\leqq0}$-comodule structure of $M$
\end{itemize}
such that $\beta$ is a morphism of $C_\zeta$-modules $($or equivalently, $\alpha$ is a morphism of $C_\zeta^{\leqq0}$-comodules$)$.
A morphism from $(M,\alpha,\beta)$ to $(M',\alpha',\beta')$ is a linear map $\varphi:M\to M'$ which is a morphism of $C_\zeta$-modules as well as that of $C_\zeta^{\leqq0}$-comodules.

An object of $\CM_\zeta^{eq}$  is a quadruple $(M,\alpha,\beta,\gamma)$ with
\begin{itemize}
\item[(1)]
$M$ is a vector space over $\BC$, 
\item[(2)]
$\alpha:C_\zeta\otimes M\to M$ is a left $C_\zeta$-module structure of $M$, 
\item[(3)]
$\beta:M\to C_\zeta^{\leqq0}\otimes M$ is a left $C_\zeta^{\leqq0}$-comodule structure of $M$,
\item[(4)]
$\gamma:M\to M\otimes C_\zeta$ is a right $C_\zeta$-comodule structure of $M$
\end{itemize}
subject to the conditions that
$(M,\alpha,\beta)\in\CM_\zeta$,
$\beta$ and $\gamma$ commutes with each other, and $\gamma$ is a homomorphism of left $C_\zeta$-modules.
A morphism from $(M,\alpha,\beta,\gamma)$ to $(M',\alpha',\beta',\gamma')$ is a linear map $\varphi:M\to M'$ which is compatible with the left $C_\zeta$-module structure,  the left $C_\zeta^{\leqq0}$-comodule structure and the right $C_\zeta$-comodule structure.

For a coalgebra $\CC$ we denote by $\Comod(\CC)$ (resp.\ $\Comod^r(\CC)$) 
the category of left $\CC$-comodules (resp.\ right $\CC$-comodules).
We define functors
\begin{align*}
&\Xi:\CM^{eq}_\zeta\to\Comod(C_\zeta^{\leqq0}),\\
&\Upsilon:\Comod(C_\zeta^{\leqq0})\to\CM^{eq}_\zeta
\end{align*}
by
\begin{align*}
&\Xi(M)=\{M\in M\mid\gamma(m)=m\otimes 1\},\\
&\Upsilon(L)=C_\zeta\otimes L.
\end{align*}
By Backelin-Kremnizer \cite{BK1} we have
\begin{proposition}
\label{prop:equivalence-O-equiv}
The functor $\Xi:\CM^{eq}_\zeta\to\Comod(C_\zeta^{\leqq0})$ gives an equivalence of categories, and its quasi-inverse is given by $\Upsilon$.
\end{proposition}
\begin{remark}
\label{rem:fiber}
{\rm
For $M\in\CM^{eq}_\zeta$ we have an isomorphism
\[
\Xi(M)\cong \BC\otimes_{C_\zeta}M
\]
of vector spaces by Proposition \ref{prop:equivalence-O-equiv}.
Here $C_\zeta\to\BC$ is given by $\varepsilon$.
}
\end{remark}

For $\lambda\in\Lambda$ we define $\chi^{\leqq0}_\lambda\in C^{\leqq0}_\zeta\subset\Hom_\BC(U_\zeta^{L,\leqq0},\BC)$ by
\[
\chi^{\leqq0}_\lambda(hu)=\chi_\lambda(h)\varepsilon(u)
\qquad
(h\in U_\zeta^{L,0}, u\in U_\zeta^{L,-}).
\]
We define left exact functors
\begin{align}
\label{eq:omega-M}
&\omega_{\CM*}:\CM_\zeta\to\Mod_\Lambda(A_\zeta),\\
\label{eq:Gamma-M}
&\Gamma_{\CM}:\CM_\zeta\to\Mod(\BC)
\end{align}
by
\begin{align*}
\omega_{\CM*}(M)&=\bigoplus_{\lambda\in\Lambda}(\omega_{\CM*}(M))(\lambda)\subset M,
\\
(\omega_{\CM*}(M))(\lambda)&=\{m\in M\mid\beta(m)=\chi_\lambda^{\leqq0}\otimes m\},\\
\Gamma_{\CM}(M)&=(\omega_{\CM*}(M))(0).
\end{align*}

We denote by $\Mod_\Lambda^{eq}(A_\zeta)$ the category consisting of $N\in\Mod_\Lambda(A_\zeta)$ equipped with a right $C_\zeta$-comodule structure $\gamma:N\to N\otimes C_\zeta$ such that 
$\gamma(N(\lambda))\subset N(\lambda)\otimes C_\zeta$ for any $\lambda\in\Lambda$ and $\gamma(\varphi n)=\Delta(\varphi)\gamma(n)$ for any $\varphi\in A_\zeta$ and $n\in N$ (note that $\Delta(A_\zeta(\lambda))\subset A_\zeta(\lambda)\otimes C_\zeta$).
By definition \eqref{eq:omega-M}, 
\eqref{eq:Gamma-M}
induce left exact functors
\begin{align}
\label{eq:omega-M-eq}
&\omega^{eq}_{\CM*}:\CM^{eq}_\zeta\to\Mod^{eq}_\Lambda(A_\zeta),\\
\label{eq:Gamma-M-eq}
&\Gamma^{eq}_{\CM}:\CM^{eq}_\zeta\to\Comod^r(C_\zeta).
\end{align}
We also define a left exact functor
\begin{equation}
\Ind:
\Comod(C_\zeta^{\leqq0})
\to
\Comod^r(C_\zeta).
\end{equation}
by
$\Ind=\Gamma^{eq}_\CM\circ\Upsilon$.

The abelian categories $\CM_\zeta$, $\CM^{eq}_\zeta$, $\Comod^r(C_\zeta)$ have enough injectives, and the forgetful functor $\CM^{eq}_\zeta\to\CM_\zeta$ sends injective objects to $\Gamma_\CM$-accyclic objects (see Backelin-Kremnizer \cite[3.4]{BK1}).
Hence we have the following.
\begin{lemma}
\label{lem:omega-eq}
We have
\begin{align*}
\For\circ R^i\Gamma^{eq}_\CM
=&R^i\Gamma_\CM\circ\For:
\CM^{eq}_\zeta\to
\Mod(\BC),\\
R^i\Ind\circ\Xi =&R^i\Gamma^{eq}_{\CM}:\CM_\zeta^{eq}\to\Comod^r(C_\zeta).
\end{align*}
for any $i$, where $\For:\Comod^r(C_\zeta)\to\Mod(\BC)$ and $\For:\CM^{eq}_\zeta\to\CM_\zeta$ are forgetful functors.
\end{lemma}

We define an exact functor
\begin{equation}
\res:\Comod^r(C_\zeta)\to\Comod(C_\zeta^{\leqq0})
\end{equation}
as follows.
For $V\in\Comod^r(C_\zeta)$ with right $C_\zeta$-comodule structure $\beta:V\to V\otimes C_\zeta$ we have $\res(V)=V$ as a $\BC$-module and the left $C_\zeta^{\leqq0}$-comodule structure 
$\res(V)\to C_\zeta^{\leqq0}\otimes\res(V)$
of $\res(V)$ is given by 
\[
\beta(v)=\sum_kv_k\otimes\varphi_k\quad
\Longrightarrow
\quad
\gamma(v)=\sum_k\res(S^{-1}\varphi_k)\otimes v_k.
\]
The following fact is standard.
\begin{lemma}
\label{lem:Ind-formula}
For $V\in\Comod^r(C_\zeta)$, $M\in\Comod(C_\zeta^{\leqq0})$ we have an isomorphism
\[
F:\Ind(M)\otimes V\to\Ind(\res(V)\otimes M)
\]
of right $C_\zeta$-comodules given by
\[
F((\sum_i\varphi_i\otimes m_i)\otimes v)=
\sum_{i,(v)}\varphi_i v_{(1)}\otimes v_{(0)}\otimes m_i,
\]
where we write the right $C_\zeta$-comodule structure of $V$ by
\[
V\ni v\mapsto\sum_{(v)}v_{(0)}\otimes v_{(1)}\in V\otimes C_\zeta.
\]
\end{lemma}
For $\lambda\in\Lambda$ we denote by $\BC^{\leqq0}_\lambda=\BC1_\lambda^{\leqq0}$ the object of $\Comod(C_\zeta^{\leqq0})$ corresponding to the one-dimensional right $U_\zeta^{L,\leqq0}$-module given by $1_\lambda^{\leqq0} u=\chi^{\leqq0}_\lambda(u)1^{\leqq0}_\lambda$ for $u\in U_\zeta^{L,\leqq0}$.
By definition 
we have an isomorphism
\[
\Ind(\BC^{\leqq0}_{-\lambda})\cong A_\zeta(\lambda)
\qquad(\lambda\in\Lambda^+)
\]
of right $C_\zeta$-comodules.

Let $N\in \Mod_\Lambda(A_\zeta)$.
Then $C_\zeta\otimes_{A_\zeta}N$ turns out to be an object of $\CM_\zeta$ by
\begin{align*}
\alpha(f\otimes(f'\otimes n)&=ff'\otimes n\qquad
(f, f'\in C_\zeta,\, n\in N),
\\
\beta(f\otimes n)&=\sum_{(f)}\res(f_{(0)})\chi_\lambda\otimes (f_{(1)}\otimes n)
\qquad(f\in C_\zeta,\,\, n\in N(\lambda)).
\end{align*}
Hence we have a functor $\Mod_\Lambda(A_\zeta)\to\CM_\zeta$ sending $N$ to $C_\zeta\otimes_{A_\zeta}N$.
\begin{lemma}
\label{lem:Phi}
The functor $\Mod_\Lambda(A_\zeta)\to\CM_\zeta$ as above induces a functor 
\[
\Phi:\Mod(\CO_{\CB_\zeta})\to\CM_\zeta.
\]
\end{lemma}
\begin{proof}
It is sufficient to show $C_\zeta\otimes_{A_\zeta}{A_\zeta}/{A_\zeta}(\lambda+\Lambda^+)=\{0\}$ for any $\lambda\in\Lambda$.
Hence we have only to  show $C_\zeta A_\zeta(\lambda)=C_\zeta$ for any $\lambda\in\Lambda^+$.
Take $\varphi\in A_\zeta(\lambda)$ such that $\varepsilon(\varphi)=1$.
We have $\Delta(A_\zeta(\lambda))\subset A_\zeta(\lambda)\otimes C_\zeta$ and hence we can write $\Delta(\varphi)=\sum_i\varphi_i\otimes\varphi'_i$ with $\varphi_i\in A_\zeta(\lambda)$, $\varphi'_i\in C_\zeta$.
Then we have $C_\zeta A_\zeta(\lambda)\ni\sum_i(S^{-1}\varphi'_i)\varphi_i=1$.
\end{proof}
We set
\[
\Psi=\omega^*\circ\omega_{\CM*}:
\CM_\zeta\to\Mod(\CO_{\CB_\zeta}).
\]
Backelin-Kremnizer \cite{BK1} obtained the following result 
using a result of Artin-Zhang \cite{AZ}.
\begin{proposition}
\label{prop:equivalence-O}
The functor $\Phi:\Mod(\CO_{\CB_\zeta})\to\CM_\zeta$ gives an equivalence of categories, and its quasi-inverse is given by $\Psi$.
Moreover, we have an identification
\[
\omega_{\CM*}\circ\Phi =\omega_*:\Mod(\CO_{\CB_\zeta})
\to\Mod_\Lambda(A_\zeta).
\]
of functors.
\end{proposition}
Hence we have the following.
\begin{lemma}
\label{lem:Gamma}
We have
\begin{equation*}
R^i\Gamma=R^i\Gamma_{\CM}\circ\Phi:\Mod(\CO_{\CB_\zeta})\to\Mod(\BC)
\end{equation*}
for any $i$.
\end{lemma}

We set
\[
\Mod^{eq}(\CO_{\CB_\zeta})=
\Mod_\Lambda^{eq}(A_\zeta)
/\Mod_\Lambda^{eq}(A_\zeta)
\cap\Tor_{\Lambda^+}(A_\zeta).
\]
Let $N\in\Mod_\Lambda^{eq}(A_\zeta)$.
We denote the right $C_\zeta$-comodule structure of $N$ by $\gamma':N\to N\otimes C_\zeta$.
Then we have a right  $C_\zeta$-comodule structure 
$\gamma:C_\zeta\otimes_{A_\zeta}N
\to(C_\zeta\otimes_{A_\zeta}N)\otimes C_\zeta$
of 
$C_\zeta\otimes_{A_\zeta}N$ given by 
\[
\gamma'(n)=\sum_kn_k\otimes\varphi_k\quad
\Longrightarrow
\quad
\gamma(f\otimes n)=\sum_{k,(f)}(f_{(0)}\otimes n_k)\otimes f_{(1)}\varphi_k\quad
\]
This gives a functor $\Mod^{eq}_\Lambda(A_\zeta)\to
\CM^{eq}_\zeta$.
Hence by Lemma \ref{lem:Phi} we have a functor
\begin{equation}
\Phi^{eq}:\Mod^{eq}(\CO_{\CB_\zeta})\to\CM^{eq}_\zeta
\end{equation}
induced by $\Phi$.
Let $M\in\CM_\zeta^{eq}$. 
The right $C_\zeta$-comodule structure of $M$ restricts to that of 
$\omega_{\CM*}M$ so that $\omega_{\CM*}M\in\Mod_\Lambda^{eq}(A_\zeta)$.
Hence we have a functor 
\begin{equation}
\Psi^{eq}:\CM^{eq}_\zeta\to\Mod^{eq}(\CO_{\CB_\zeta})
\end{equation}
induced by $\Psi$.
By Proposition \ref{prop:equivalence-O} we have the following.
\begin{proposition}
\label{prop:equivalence-O-eq}
The functor $\Phi^{eq}:\Mod^{eq}(\CO_{\CB_\zeta})\to\CM^{eq}_\zeta$ gives an equivalence of categories, and its quasi-inverse is given by $\Psi^{eq}$.
\end{proposition}

By Proposition \ref{prop:equivalence-O-eq} we see that 
\eqref{eq:omega-M}, \eqref{eq:Gamma-M} induce
\begin{align}
\label{eq:Phieq1}
&\omega^{eq}_*=\omega_{\CM*}^{eq}\circ\Phi^{eq}:\Mod^{eq}(\CO_{\CB_\zeta})\to\Mod_\Lambda^{eq}(A_\zeta),\\
\label{eq:Phieq2}
&
\Gamma^{eq}
=\Gamma_{\CM}^{eq}\circ\Phi^{eq}:\Mod^{eq}(\CO_{\CB_\zeta})\to
\Comod^r(C_\zeta).
\end{align}

By Lemma \ref{lem:omega-eq} we have the following.
\begin{lemma}
\label{lem:Gamma1}
We have
\[
\For\circ R^i\Gamma^{eq}
=R^i\Gamma\circ\For:
\Mod^{eq}(\CO_{\CB_\zeta})\to
\Mod(\BC)
\]
for any $i$, where $\For:\Comod^r(C_\zeta)\to\Mod(\BC)$ and $\For:\Mod^{eq}(\CO_{\CB_\zeta})\to\Mod(\CO_{\CB_\zeta})$ are forgetful functors.
\end{lemma}

\section{Reformulation of Conjecture \ref{conj2}}
\subsection{Adjoint action of $U_\zeta^L$ on $D'_{\zeta}$}
Define a left $U_\BF$-module structure of $E_\BF$ by
\begin{align*}
\ad(u)(P)&=\sum_{(u)}u_{(0)}P(Su_{(1)})
\qquad
(u\in U_\BF,\, P\in E_\BF).
\end{align*}
Then we have
\[
\ad(u)(P_1P_2)=\sum_{(u)}\ad(u_{(0)})(P_1)\ad(u_{(1)})(P_2)\qquad
(P_1, P_2\in E_\BF),
\]
\begin{align*}
\ad(u)(\varphi)&=u\cdot\varphi
\quad(\varphi\in A_\BF\subset E_\BF),\\
\ad(u)(v)&=\sum_{(u)}u_{(0)}v(Su_{(1)})
\quad(v\in U_\BF\subset E_\BF),\\
\ad(u)(e(\lambda))&=\varepsilon(u)e(\lambda)
\quad(\lambda\in\Lambda, e(\lambda)\in\BF[\Lambda]\subset E_\BF)
\end{align*}
for $u\in U_\BF$.
We see from \cite[Lemma 4.2]{TI} that 
this induces
a left $U_\BF$-module structure of $D'_\BF$.
Moreover, the $U_\BF$-module structures of $E_\BF$ and $D'_\BF$ induce $U_\BA^L$-module structures of $E_\BA$, $D'_\BA$,  $E_{\BA,\diamondsuit}$, $D'_{\BA,\diamondsuit}$, $E_{\BA,f}$, $D'_{\BA,f}$ by Lemma \ref{lem:ad} and Lemma \ref{lem:diamondsuit}.
Hence by specialization we obtain $U_\zeta^L$-module structures of $E_\zeta$, $D'_\zeta$,  $E_{\zeta,\diamondsuit}$, $D'_{\zeta,\diamondsuit}$, $E_{\zeta,f}$, $D'_{\zeta,f}$ also denoted by $\ad$.

\subsection{}
We will regard $E_{\zeta,f}, D'_{\zeta,f}\in\Mod_\Lambda(A_\zeta)$ as  objects of $\Mod_\Lambda^{eq}(A_\zeta)$
by the right $C_\zeta$-comodule structures induced from the left $U^L_\zeta$-module structures
\[
(u,P)\mapsto \ad(u)(P)
\qquad(u\in U_\zeta^L, P\in E_{\zeta,f} \text{ or } D'_{\zeta,f}).
\]
Then for
\[
\left(\Xi\circ\Phi^{eq}\right)(\omega^*D'_{\zeta,f})\in\Comod(C^{\leqq0}_\zeta)
\]
we have
\[
R^i\Gamma(\omega^*D'_{\zeta,f})
=R^i\Ind(\left(\Xi\circ\Phi^{eq}\right)(\omega^*D'_{\zeta,f}))
\]
by Lemma \ref{lem:omega-eq}, Lemma \ref{lem:Gamma1}, and \eqref{eq:Phieq2}.

Define a right $(U_{\zeta,\diamondsuit}\otimes\BC[\Lambda])$-module $V$ by
\[
V=(U_{\zeta,\diamondsuit}\otimes\BC[\Lambda])/\CI
\]
where
\[
\CI=
(\tilde{U}_\zeta^-\cap\Ker(\varepsilon))U_{\zeta,\diamondsuit}\BC[\Lambda]
+\sum_{\lambda\in\Lambda}(k_{2\lambda}-e({2\lambda}))U_{\zeta,\diamondsuit}\BC[\Lambda].
\]
By the triangular decomposition $\tilde{U}_\zeta^-\otimes U_{\zeta,\diamondsuit}^0\otimes U_\zeta^+\cong U_{\zeta,\diamondsuit}$ 
we have
\[
V\cong U_\zeta^+\otimes\BC[\Lambda]
\]
as a vector space.
Define a right action of $U_\zeta^{L,\leqq0}$ on $U_{\zeta,\diamondsuit}\otimes\BC[\Lambda]$ by 
\[
(u\otimes e(\lambda))\star v
=\ad(Sv)(u)\otimes e(\lambda)
\qquad(u\in U_{\zeta,\diamond}, \lambda\in\Lambda, v\in U_\zeta^{L\leqq0}).
\]
It induces a right action of $U_\zeta^{L,\leqq0}$ on $V$.
Moreover, we see easily that this right $U_\zeta^{L,\leqq0}$-module structure gives a left $C_\zeta^{\leqq0}$-comodule structure of $V$.
\begin{proposition}
\label{prop:RiInd}
We have
\[
\left(\Xi\circ\Phi^{eq}\right)(\omega^*D'_{\zeta,f})
\cong
V
\]
as a left $C_\zeta^{\leqq0}$-comodule.
\end{proposition}
The proof is given in the next subsection.

It follows from Proposition \ref{prop:RiInd} that Conjecture \ref{conj2} is equivalent to the following conjecture.
\begin{conjecture}
\label{conj3}
Assume $\ell>h_G$.
We have
\[
\Ind(V)\cong 
U_{\zeta,f}\otimes_{Z_{Har}(U_\zeta)}\BC[\Lambda],
\]
and
\[
R^i\Ind(V)=0
\]
for $i\ne0$.
\end{conjecture}
\begin{remark}
{\rm
We can show that 
\[
U_{\zeta,f}\cong  (C_\zeta)_{\ad},\qquad
V\cong {}_{\ad}(C_\zeta^{\leqq0})\otimes_{\BC[2\Lambda]}\BC[\Lambda],
\]
where 
$(C_\zeta)_{\ad}$ (resp.\ ${}_{\ad}(C_\zeta^{\leqq0})$) 
is given by the right (resp.\ left) adjoint coaction of 
$C_\zeta$ (resp.\ $C_\zeta^{\leqq0}$) 
on itself.
Hence Conjecture \ref{conj3} is equivalent to 
\[
R\Ind({}_{\ad}(C_\zeta^{\leqq0}))
\cong
(C_\zeta)_{\ad}\otimes_{\BC[2\Lambda]^W}\BC[2\Lambda].
\]
The corresponding statement for $q=1$ is 
\[
R\Ind({}_{\ad}\BC[B^-])\cong{\BC[G]}_{\ad}\otimes_{\BC[H/W]}\BC[H].
\]
We can prove this by a geometric method.
}
\end{remark}

\begin{remark}\footnote{this remark is added at the editor's request.}
\label{rem:BK}
{\rm
A proof of Conjecture \ref{conj3}, when $\ell$ is a prime greater than the Coxeter number,  is given in Backelin-Kremnizer \cite{BK2}; however, in a more recent paper \cite{BK3}
they admit that there are gaps in \cite{BK2}
(see 1.1.2 of \cite[Version 3]{BK3}), and propose different proofs. 
But it is likely that there still remain problems in the new proofs given in \cite{BK3} as explained below.

The proof in \cite[Version 1, Version 2]{BK3} is wrong since all positive roots are assumed there to be dominant (see the  proof of Theorem 2.1 in \cite[Version 2]{BK3}).

Another proof given in \cite[Version 3]{BK3} also has problems.
In Step b)  of the proof of \cite[Version 3, Theorem 2.2.1]{BK3}, the authors compare certain weight multiplicities $a_{q,\mu}$ and $b_{q,\mu}$.
But since those multiplicities are infinite, the argument there should be modified using multplicities as $U_q$-modules.
Let us assume for simplicity that $q$ is generic and 
try to modify the original argument 
by replacing $a_{q,\mu}, b_{q,\mu}, b'_{q,\mu}$ with their counterparts as multiplicities of  $U_q$-modules.
This even fails since $a_{1,\mu}$ (resp. $b'_{1,\mu}$) is the dimension of the zero weight space of the irreducible modules (resp. Verma module) with highest weight $\mu$.
We also point out that the reason for $U_q^\lambda$ to be an integral domain is not given
in Step a).

Note that the arguments in the proof of \cite[Version 3, Theorem 2.2.1]{BK3}
are partially similar to those in the earlier manuscripts \cite[Proposition 4.8]{BK1}, \cite[Proposition 3.25]{BK2}.
The main difference is that 
\cite[Version 3]{BK3} relies on a $B_q$-stable filtration with one-dimensional subquotients instead of the Joseph-Letzter filtration used in \cite{BK1, BK2}.
For us the original argument in \cite{BK1, BK2} for generic $q$ using the Joseph-Letzter filtration
is not comprehensible either.
In the notation of the proof of \cite[Proposition 4.8]{BK1} the validity of the formula 
$m_j(1)=\tilde{n}_j(1)$ is not clear to us since the Joseph-Letzter filtration
does not induce at $q=1$ 
the ordinary filtration for enveloping algebras and differential operators in general.
}
\end{remark}

\subsection{}
We will give a proof of Proposition \ref{prop:RiInd} in the rest of this paper.
By Remark \ref{rem:fiber}.
we have
\[
\left(\Xi\circ\Phi^{eq}\right)(\omega^*D'_{\zeta,f})
\cong
\BC\otimes_{A_\zeta}D'_{\zeta,f}
\]
as a vector space, where $A_\zeta\to\BC$ is given by $\varepsilon$.
Note that 
\[
\BC\otimes_{A_\zeta} E_{\zeta,\diamondsuit}\cong U_{\zeta,\diamondsuit}\otimes\BC[\Lambda].
\]
We first show the following.
\begin{lemma}
\label{lem:key-calculation}
We have
\[
\BC\otimes_{A_\zeta} D'_{\zeta,\diamondsuit}\cong 
V.
\]
\end{lemma}
\begin{proof}
By \eqref{eq:D-prime}
we obtain
\[
\BC\otimes_{A_\zeta}D'_{\zeta,\diamondsuit}
\cong
(U_{\zeta,\diamondsuit}\otimes\BC[\Lambda])/
\sum_{\varphi\in A_\zeta}(1\otimes\Omega'(\varphi))(U_{\zeta,\diamondsuit}\otimes\BC[\Lambda]),
\]
where $1\otimes\Omega'(\varphi)$ is the image of $\Omega'(\varphi)$ in $\BC\otimes_{A_\zeta}E_{\zeta,\diamondsuit}=U_{\zeta,\diamondsuit}\otimes\BC[\Lambda]$.
Note that $\varepsilon(A_\zeta(\lambda)_\xi)=\{0\}$ for $\lambda\in\Lambda^+$, $\xi\in\Lambda$ with $\lambda\ne\xi$, and $\varepsilon(A_\zeta(\lambda)_\lambda)=\BC$ for $\lambda\in\Lambda^+$.
Hence for $\varphi\in A_\zeta(\lambda)_\xi$ with $\lambda\in\Lambda^+$, $\xi\in\Lambda$ we have
\[
1\otimes\Omega'_1(\varphi)=
\begin{cases}
0\qquad&(\lambda\ne\xi)\\
\varepsilon(\varphi)\qquad&(\lambda=\xi).
\end{cases}
\]
Let us also compute $1\otimes \Omega'_2(\varphi)$.
Let
\[
\tilde{\Psi}_\lambda:\tilde{U}^-_\zeta(\lambda)\to 
A_\zeta(\lambda)
\]
be the composite of the linear isomorphism 
$\Psi_\lambda:\tilde{U}^-_\zeta(\lambda)
\to L^*_{-,\zeta}(\lambda)$
(see \eqref{eq:PhiPsiLzeta})
and an isomorphism $f:L^*_{-,\zeta}(\lambda)\to A_\zeta(\lambda)$ of $U_\zeta^L$-modules.
We have $\tilde{\Psi}_\lambda(\tilde{U}^-_\zeta(\lambda)_{-(\lambda-\xi)})= A_\zeta(\lambda)_\xi$ for any $\xi\in \Lambda$.
Hence we may assume $\varepsilon=\varepsilon\circ\tilde{\Psi}_\lambda$ on $\tilde{U}^-_\zeta(\lambda)$. 
Let $\varphi\in A_\zeta(\lambda)_\xi$ and take 
$v\in \tilde{U}^-_\zeta(\lambda)_{-(\lambda-\xi)}$ satisfying $\tilde{\Psi}_\lambda(v)=\varphi$.
Then we have
\begin{align*}
&\sum_p(Sx_p^L)\cdot\varphi\otimes y_pk_{\beta_p}
\\
=&\sum_pf((Sx_p^L)\cdot\Psi_\lambda(v))\otimes y_pk_{\beta_p}
\\
=&
\sum_p
\zeta^{-(\beta_p,\xi)}
f((Sx_p^L)k_{\beta_p}\cdot\Psi_\lambda(v))\otimes y_pk_{\beta_p}
\\
=&
\sum_{p,(v)}
\zeta^{-(\beta_p,\xi)}
f({}^L\tau_\zeta((Sx_p^L)k_{\beta_p},v_{(0)})
\Psi_\lambda(v_{(1)}))\otimes y_pk_{\beta_p}
\\
=&
\sum_{p,(v)}
\zeta^{-(\beta_p,\xi)}
{}^L\tau_\zeta((Sx_p^L)k_{\beta_p},v_{(0)})
\tilde{\Psi}_\lambda(v_{(1)})\otimes y_pk_{\beta_p},
\end{align*}
and hence
\begin{align*}
&
1\otimes \Omega'_2(\varphi)\\
=&
\sum_p
\varepsilon((Sx_p^L)\cdot\varphi)y_pk_{\beta_p}k_{2\xi} e(-2\lambda)
\\
=&
\sum_{p,(v)}
\zeta^{-(\beta_p,\xi)}
{}^L\tau_\zeta((Sx_p^L)k_{\beta_p},v_{(0)})
\varepsilon(v_{(1)})y_pk_{\beta_p}k_{2\xi} e(-2\lambda)
\\
=&
\sum_{p}
\zeta^{-(\beta_p,\xi)}
{}^L\tau_\zeta((Sx_p^L)k_{\beta_p},v)
y_pk_{\beta_p}k_{2\xi} e(-2\lambda)
\\
=&
\sum_{p}
\zeta^{-(\beta_p,\xi)}
{}^L\tau_\zeta(k_{-\beta_p}x_p^L,S^{-1}v)
y_pk_{\beta_p}k_{2\xi} e(-2\lambda)
\\
=&
\sum_{p}
\zeta^{-(\beta_p,\xi)-(\beta_p,\beta_p)}
{}^L\tau_\zeta(x_p^L,S^{-1}v)
y_pk_{\beta_p}k_{2\xi} e(-2\lambda)
\\
=&
\sum_{p}
\zeta^{-(\lambda-\xi,\lambda)}
{}^L\tau_\zeta(x_p^L,S^{-1}v)
y_pk_{\lambda-\xi}k_{2\xi} e(-2\lambda)
\\
=&
\zeta^{-(\lambda-\xi,\lambda)}
(S^{-1}v)k_{\lambda-\xi}k_{2\xi} e(-2\lambda)
\end{align*}
(note $(S^{-1}v)k_{\lambda-\xi}\in \tilde{U}^-_\zeta(\lambda)_{-(\lambda-\xi)}$).
It follows that
\[
1\otimes\Omega'(\varphi)=
\begin{cases}
-\zeta^{-(\lambda-\xi,\lambda)}
(S^{-1}v)k_{\lambda-\xi}k_{2\xi} e(-2\lambda)\qquad
&(\lambda\ne\xi)\\
\varepsilon(\varphi)(1-k_{2\lambda} e(-2\lambda))\qquad&(\lambda=\xi).
\end{cases}
\]
Hence we have
\begin{align*}
&\sum_{\lambda\in\Lambda^+,\atop \gamma\in Q^+}
\sum_{\varphi\in A_\zeta(\lambda)_{\lambda-\gamma}}
(1\otimes\Omega'(\varphi))(U_{\zeta,\diamondsuit}\otimes\BC[\Lambda])\\
=&
\sum_{\lambda\in\Lambda^+,\atop \gamma\in Q^+\setminus\{0\}}
\tilde{U}_\zeta^-(\lambda)_{-\gamma}
(U_{\zeta,\diamondsuit}\otimes\BC[\Lambda])
+\sum_{\lambda\in\Lambda^+}(1-k_{2\lambda} e(-2\lambda))(U_{\zeta,\diamondsuit}\otimes\BC[\Lambda])
\\
=&
(\tilde{U}_\zeta^-\cap\Ker(\varepsilon))
(U_{\zeta,\diamondsuit}\otimes\BC[\Lambda])
+\sum_{\lambda\in\Lambda}(k_{2\lambda}-e(2\lambda))(U_{\zeta,\diamondsuit}\otimes\BC[\Lambda])
\end{align*}
by \eqref{eq:generating:zeta:2}.
\end{proof}

\begin{lemma}
\label{lem:key-calculation2}
We have
\[
\BC\otimes_{A_\zeta} D'_{\zeta,f}\cong 
V.
\]
\end{lemma}
\begin{proof}
We need to show that the canonical homomorphism
$\BC\otimes_{A_\zeta} D'_{\zeta,f}
\to
\BC\otimes_{A_\zeta} D'_{\zeta,\diamondsuit}$
is bijective.
The surjectivity is a consequence of 
\eqref{eq:generating:zeta:2}, \eqref{eq:generating:zeta:3}.
Let us give a proof of the injectivity.
Set
\[
\CK
=
A_\zeta U_{\zeta, f}\BC[\Lambda]
\cap
\sum_{\varphi\in A_\zeta}
A_\zeta\Omega'(\varphi)U_{\zeta,\diamondsuit}\BC[\Lambda]
\subset
A_\zeta \otimes U_{\zeta, f}\otimes \BC[\Lambda].
\]
Then it is sufficient to show that the natural map
\[
\BC\otimes_{A_\zeta}((A_\zeta \otimes U_{\zeta, f}\otimes \BC[\Lambda])/\CK)
\to
(U_{\zeta, \diamondsuit}\otimes\BC[\Lambda])/\CI
\]
is injective.
Let $F:A_\zeta \otimes U_{\zeta, f}\otimes \BC[\Lambda]
\to
U_{\zeta, \diamondsuit}\otimes\BC[\Lambda]$
be the natural map.
Then it is sufficient to show 
\begin{equation}
\label{eq:claim}
\CI\cap (U_{\zeta,f}\otimes\BC[\Lambda])\subset F(\CK).
\end{equation}
Indeed, assume that \eqref{eq:claim} holds.
Denote by 
\begin{gather*}
p:A_\zeta \otimes U_{\zeta, f}\otimes \BC[\Lambda]
\to
\BC\otimes_{A_\zeta}((A_\zeta \otimes U_{\zeta, f}\otimes \BC[\Lambda])/\CK),\\
\pi:U_{\zeta, \diamondsuit}\otimes\BC[\Lambda]
\to
(U_{\zeta, \diamondsuit}\otimes\BC[\Lambda])/\CI
\end{gather*}
the natural maps.
We have to show $\Ker(\pi\circ F)\subset\Ker(p)$.
Take $x\in\Ker(\pi\circ F)$.
Then $F(x)\in\CI\cap(U_{\zeta,f}\otimes\BC[\Lambda])$.
Hence by \eqref{eq:claim} there exists some $v\in\CK$ such that $F(x)=F(v)$.
Then $p(x)=p(x-v)+p(v)=p(x-v)$.
Hence we may assume that $F(x)=0$ from the beginning.
Note that $p$ factors through 
\[
p':A_\zeta \otimes U_{\zeta, f}\otimes \BC[\Lambda]
\to
\BC\otimes_{A_\zeta}(A_\zeta \otimes U_{\zeta, f}\otimes \BC[\Lambda])
(=U_{\zeta, f}\otimes \BC[\Lambda]).
\]
By $F(x)=0$ we have $p'(x)=0$ and hence $p(x)=0$ as desired.

It remains to show \eqref{eq:claim}.
Let $\lambda\in\Lambda^+$ and $\varphi\in A_\zeta(\lambda)_\lambda$.
Then we have
\[
\Omega_1'(\varphi)=\sum_p(y^L_p\cdot\varphi) x_p\in A_\zeta U^+_\zeta,\qquad
\Omega_2'(\varphi)=\varphi k_{2\lambda}e(-2\lambda).
\]
Let us show
\begin{equation}
\label{eq:claim1}
\Omega_1'(\varphi)=\sum_p(y^L_p\cdot\varphi) x_p\in A_\zeta U_\zeta^+(\lambda).
\end{equation}
This is equivalent to
\[
\sum_p(y^L_p\cdot\varphi) \otimes\Phi_{-\lambda}(x_p)
\in
A_\zeta\otimes L_{+,\zeta}^*(-\lambda).
\]
This follows from 
\begin{align*}
&
\sum_p
\left\langle
\Phi_{-\lambda}(x_p),
\overline{
uf_i^{((\lambda,\alpha_i^\vee)+1)}
}
\right\rangle
y^L_p\cdot\varphi
=
\sum_p
\tau^{L}_\zeta\left(
x_p,
u
f_i^{((\lambda,\alpha_i^\vee)+1)}
\right)
y^L_p\cdot\varphi\\
=&(uf_i^{((\lambda,\alpha_i^\vee)+1)})\cdot\varphi=0
\end{align*}
for $u\in U_\zeta^{L,-}, i\in I$.
\eqref{eq:claim1} is verified.
Hence we have
\[
\Omega'(\varphi)k_{-2\lambda}\in\CK.
\]
It follows that
\begin{equation}
\label{eq:claim2}
F(\CK)\supset(k_{-2\lambda}-e(-2\lambda))U_{\zeta, f}\BC[\Lambda]
\qquad
(\lambda\in\Lambda^+).
\end{equation}
Now let $u\in \CI\cap (U_{\zeta,f}\otimes\BC[\Lambda])$.
If we can show that $k_{-2\mu}u\in F(\CK)$ for some $\mu\in\Lambda^+$, then we obtain
\[
u=e(2\mu)(e(-2\mu)-k_{-2\mu})u+
e(2\mu)k_{-2\mu}u\in F(\CK)
\]
by \eqref{eq:claim2}.
Hence it is sufficient to show that for any $u\in\CI$ there exists some $\mu\in\Lambda^+$ such that $k_{-2\mu}u\in F(\CK)$.
We may assume that there exists $\nu\in Q$ such that
$k_{-2\mu}u=\zeta^{(\mu,\nu)}uk_{-2\mu}$ for any $\mu\in\Lambda$.
Therefore, we have only to show that for any $u\in\CI$ there exists some $\mu\in\Lambda^+$ such that $uk_{-2\mu}\in F(\CK)$.
By Lemma \ref{lem:key-calculation} 
we can take $\varphi_i\in A_\zeta$, 
$x_i\in U_{\zeta,\diamondsuit}\otimes\BC[\Lambda]$ ($i=1,\dots, N$) such that 
\[
u=1\otimes\sum_{i=1}^N\Omega'(\varphi_i)x_i.
\]
By Lemma \ref{lem:localization}
we can take $\mu\in\Lambda^+$ such that 
$\Omega'(\varphi_i)x_ik_{-2\mu}\in A_\zeta\otimes U_{\zeta,f}\otimes\BC[\Lambda]$ for any $i$.
Then we have
\[
uk_{-2\mu}
=\sum_{i=1}^N
F(\Omega'(\varphi_i)x_ik_{-2\mu})\in
F(\CK).
\]
\end{proof}
By Lemma \ref{lem:key-calculation2}
we obtain an isomorphism
\[
\left(\Xi\circ\Phi^{eq}\right)(\omega^*D'_{\zeta,f})
\cong
V
\]
of vector spaces.
We need to show that it is in fact an isomorphism of left $C_\zeta^{\leqq0}$-comodules.
This is a consequence of the corresponding fact for $E_{\zeta,f}$.
Note that
we have
\[
\BC\otimes_{A_\zeta}E_{\zeta,f}
\cong U_{\zeta,f}\otimes\BC[\Lambda],
\]
and hence we have an isomorphism
\begin{equation}
\label{eq:ad-action}
\left(\Xi\circ\Phi^{eq}\right)(\omega^*E_{\zeta,f})
\cong
U_{\zeta,f}\otimes\BC[\Lambda]
\end{equation}
of vector spaces.
Hence we have only to show the following.
\begin{lemma}
Under the identification \eqref{eq:ad-action} the left $C_\zeta^{\leqq0}$-comodule structure of $U_{\zeta,f}\otimes\BC[\Lambda]$ is associated to the right $U_\zeta^{L,\leqq0}$-module structure given by
\[
(u\otimes e(\lambda))\cdot v
=
\ad(Sv)(u)\otimes e(\lambda)
\qquad
(u\in U_{\zeta,f}, \lambda\in\Lambda, v\in U_\zeta^{L,\leqq0}).
\]
\end{lemma}
\begin{proof}
Note that the left $C_\zeta^{\leqq0}$-comodule structure of
$U_{\zeta,f}\otimes\BC[\Lambda]$ is given by
\[
U_{\zeta,f}\otimes\BC[\Lambda]
\cong\Xi(C_\zeta\otimes (U_{\zeta,f}\otimes\BC[\Lambda])),
\]
where $C_\zeta\otimes (U_{\zeta,f}\otimes\BC[\Lambda])$ is regarded as a left $C_\zeta^{\leqq0}$-comodule by the tensor product of $C_\zeta$ (with left $C_\zeta^{\leqq0}$-comodule structure $(\res\otimes 1)\circ\Delta:C_\zeta\to C_\zeta^{\leqq0}\otimes C_\zeta$) and $U_{\zeta,f}\otimes\BC[\Lambda]$ with trivial left $C_\zeta^{\leqq0}$-comodule structure.
Hence it is sufficient to show that for a right $C_\zeta$-comodule $M$ 
the right $U_\zeta^{L,\leqq0}$-module structure of 
\[
M\cong\Xi(C_\zeta\otimes M)\in\Comod(C_\zeta^{\leqq0})
\]
is given by
\[
m\cdot v
=(Sv)\cdot m
\quad
(m\in M, 
v\in U_\zeta^{L,\leqq0}).
\]
Denote by $M^{triv}$ the trivial right $C_\zeta$-comodule which coincides with $M$ as a vector space. 
We denote by $M\ni m\leftrightarrow \overline{m}\in M^{triv}$ the canonical linear isomorphism.
We have $C_\zeta\otimes M^{triv}\in\Comod^r(C_\zeta)$ as the tensor product of $C_\zeta\in\Comod^r(C_\zeta)$ and $M^{triv}\in\Comod^r(C_\zeta)$. 
We can also define a left $C_\zeta^{\leqq0}$-comodule structure of $C_\zeta\otimes M^{triv}$ as the tensor product of the left $C_\zeta^{\leqq0}$-comodules $C_\zeta$ and $M^{triv}$ where the left $C_\zeta^{\leqq0}$-comodule structure of $M^{triv}$ is given by the right $U_\zeta^{L,\leqq0}$-module structure 
\[
\overline{m}\cdot v=\overline{(Sv)\cdot m}
\qquad
(m\in M, 
v\in U_\zeta^{L,\leqq0}).
\]
Then we have an linear isomorphism 
\[
C_\zeta\otimes M
\ni\varphi\otimes m
\mapsto\sum_{(m)}\varphi m_{(1)}\otimes \overline{m_{(0)}}
\in C_\zeta\otimes M^{triv}
\]
preserving the right $C_\zeta$-comodule structures and the left $C_\zeta^{\leqq0}$-comodule structures.
It follows that 
\[
\Xi(C_\zeta\otimes M)\cong
\Xi(C_\zeta\otimes M^{triv})=M^{triv}\in\Comod(C_\zeta^{\leqq0}).
\]

\end{proof}
The proof of Proposition \ref{prop:RiInd} is complete.

\bibliographystyle{unsrt}

\end{document}